\documentclass[reqno]{amsart}
\usepackage{amsmath,amsthm,amssymb}
\usepackage{latexsym}
\usepackage{eucal}
\usepackage{fullpage}
\usepackage{amsmath}
\numberwithin{equation}{section}

\newtheorem{theorem}{Theorem}[section]
\newtheorem{lemma}{Lemma}[section]

\def\cal{\mathcal}
\let\Re=\undefined
\DeclareMathOperator{\Re}{Re}
\let\Im=\undefined
\DeclareMathOperator{\Im}{Im}

\def\ge{\geqslant}\def\le{\leqslant}
\def\e{\varepsilon}\def\~{\widetilde}

\begin{document}
\title[ On a problem by Steklov ]{ On a problem by Steklov }
\author{A. Aptekarev, S. Denisov, D. Tulyakov }
\address{
\begin{flushleft}
University of Wisconsin--Madison\\  Mathematics Department\\
480 Lincoln Dr., Madison, WI, 53706, USA\\  denissov@math.wisc.edu\\ \vspace{0.5cm}
Keldysh Institute for Applied Mathematics, Russian Academy of Sciences\\
Miusskaya pl. 4, 125047 Moscow, RUSSIA\\aptekaa@keldysh.ru
\end{flushleft}
}\maketitle

\begin{abstract}
Given any $\delta\in(0,1)$, we define the Steklov class $S_\delta$ to be the set of
probability measures $\sigma$ on the unit circle $\mathbb T$, such that
$\sigma'(\theta)\ge\delta/(2\pi)>0$ at every Lebesgue point of $\sigma$.
One can define the orthonormal polynomials $\phi_n(z)$ with respect to $\sigma\in S_\delta$.
In this paper, we obtain the sharp estimates on the uniform norms $\|\phi_n\|_{L^\infty(\mathbb T)}$
as $n\to\infty$ which settles a question asked by Steklov in 1921. As an important intermediate step,
we consider the following variational problem. Fix $n\in\mathbb N$ and define
$M_{n,\delta}=\smash{\sup\limits_{\sigma\in S_\delta}}\|\phi_n\|_{L^\infty(\mathbb T)}$.
Then, we prove
\[
C(\delta)\sqrt n < M_{n,\delta}\le \sqrt{\frac{n+1}\delta}\,\,.
\]
A new method is developed that can be used to study other important
variational problems. For instance, we prove the sharp estimates for
the polynomial entropy in the Steklov class.
\end{abstract}\vspace{1cm}

 {\Large\part*{Introduction.}}\bigskip

\large One version of the Steklov's problem  (see~\cite{1}, \cite{2}) is to obtain the bounds
on the sequence of polynomials $\{P_{n}(x)\}_{n=0}^{\infty}$, which are orthonormal
\begin{equation}\label{1}
\int^1_{-1}P_n(x)\,P_m(x)\,\rho(x)\,dx=\delta_{n,m}\;,\quad n,m=0,1,2\,\ldots
\end{equation}
with respect to the strictly positive weight $\rho$:
\begin{equation}\label{2}
\rho(x)\ge\delta>0\;,\quad x\in[-1,1].
\end{equation}
In 1921, V.A.~Steklov made a conjecture that a sequence $\{P_n(x)\}$ is bounded at any point
$x_{}{\in}(-1,1)$, i.e.,\begin{equation}\label{obs}
\limsup_{n\to\infty}|P_n(x)|<\infty
\end{equation}
provided that the weight $\rho$ does not vanish on $[-1,1]$. On page 321, he writes
(adapted translation from French):\smallskip

{\it ``I believe that inequality \eqref{obs} is the common property of all polynomials whose
orthogonality weight $\rho$ does not vanish inside the given interval, but so far I haven't
succeeded in finding either the rigorous proof to that statement or an example when this
estimate does not hold at each interior point of the given interval".\smallskip}

This problem and some related questions gave rise to extensive
research, see, e.g., \cite{Ger1,Ger2, Ger3, Gol} and the survey
\cite{2} for a detailed discussion and the list of references. In
1979, Rakhmanov~\cite{3} disproved this conjecture by constructing a
weight from the Steklov class \eqref{2}, for which
$$
\limsup_{n\to\infty}|P_n(0)|=\infty\;.$$
It is known (see, for example~\cite{5}) that the following bound
$$
|P_n(x)|=o(\sqrt n)$$ holds for any $x\in(-1,1)$ as long as $\rho$
satisfies \eqref{2}. In his next paper~\cite{4}, Rakhmanov proved
that for every $\e>0$ and $x_0\in(-1,1)$ there is a weight
$\rho(x;x_0,\e)$ from the Steklov class such that the corresponding
$\{P_n(x)\}$ grow as
\begin{equation}\label{3}
|P_{k_n}(x_0)|\ge k_n^{1/2-\e}\;,\quad\end{equation}
where $\{k_n\}$ is some subsequence in $\mathbb N$. In \cite{murman},
the size of the polynomials was studied for the continuous weight.\smallskip

All Rakhmanov's counterexamples were obtained as corollaries of the corresponding results
for the polynomials $\{\phi_n\}$ orthonormal on the unit circle
\begin{eqnarray}\label{4}
\int_0^{2\pi}\phi_n(e^{i\theta},\sigma)\,\overline{\phi_m(e^{i\theta},\sigma)}\,d\sigma(\theta)=
\delta_{n,m}\;,\quad
n,m=0,1,2\,\ldots\,,\\\phi_n(z,\sigma)=\lambda_nz^n+\ldots,\quad
\lambda_n>0\nonumber
\end{eqnarray}
with respect to measures from the Steklov class $S_{\delta}$ defined as the class of
probability measures $\sigma$ on the unit circle satisfying
\begin{equation}\label{gabor}\sigma'\ge\delta/(2\pi)\end{equation}
at every Lebesgue point. The version of Steklov's conjecture for
this situation would be to prove that the sequence
$\{\phi_n(z,\sigma)\}$ is bounded in $n$ at every $z{\in}\mathbb T$
provided that $\sigma{\in}S_\delta$.\smallskip

This conjecture might be motivated by the following estimate.
Consider the Christoffel-Darboux kernel
\[
K_n(\xi,z,\mu)=\sum^n_{j=0}\overline{\phi_j(\xi,\mu)}\phi_j(z,\mu)
\]
for $\xi=z$ as the function of $\mu$. If $\mu_1\le\mu_2$, then (see
\cite{5} or \cite{6})
\begin{equation}\label{monotona}
K_n(z,z,\mu_2)\le K_n(z,z,\mu_1),\quad z\in\mathbb T\,.
\end{equation}
Here we do not assume $\mu_{1(2)}$ to be probability measures, of course. Therefore, if
$\sigma{\in}S_\delta$, we get
\begin{equation}\label{chezaro}
K_n(z,z,\sigma)\le\frac{n+1}\delta,\quad\text{i.e.}\quad\frac1{n+1}\sum_{j=0}^n
|\phi_j(z,\sigma)|^2\le\delta^{\,-1}\, .
\end{equation}
by taking $\mu_1=\delta(2\pi)^{-1} d\theta$ in \eqref{monotona}.

So, on average the polynomials $\phi_n(z,\sigma)$ are indeed bounded
in $n$ and one might want to know whether they are bounded for all
$n$. Rakhmanov proved the following Theorem which gave a negative
answer to this question.\bigskip
\begin{theorem}{\cite{4}}\label{rrra}
Let $\sigma{\in}S_\delta$, where $\delta$ is sufficiently small. Then,
for every sequence $\{\beta_n\}: \beta_n{\to}0$, there is $\sigma{\in}S_\delta$ such that
\begin{equation}\label{est-ra}
\|\phi_{k_n}(z,\sigma)\|_{L^\infty(\mathbb T)}>\beta_{k_n}\sqrt{\frac{k_n}{\ln^3 k_n}}\end{equation}
for some sequence $\{k_n\}\subset\mathbb N$.
\end{theorem}\bigskip

This estimate is almost sharp due to the following result  (see,
e.g., \cite{neva}, p.11 for the real line case; \cite{5}, p. 32,
theorem~3.5 for the pointwise estimate).
\begin{theorem}\label{neva-ger}
For $\sigma\in S_\delta$, we have
\begin{equation}\label{malenkoeo}
\|\phi_n(z,\sigma)\|_{L^\infty(\mathbb{T})}={o}(\sqrt{n})\, .
\end{equation}
\end{theorem}\noindent
(for completeness, we give the proof in the end of Appendix A).\smallskip

In the proof of the Theorem \ref{rrra}, an important role was played
by the following extremal problem. For a fixed $n$, define
\begin{equation}\label{6}
M_{n,\delta}=\sup\limits_{\sigma\in S_{\delta}}\|\phi_n(z,\sigma)\|_{L^\infty(\mathbb T)}\, .
\end{equation}

One of the key results in~\cite{4} is the following inequality
\begin{equation}\label{7}
C\,\sqrt{\frac{n+1}{\delta\ln^3 n}}\le M_{n,\delta}\;,\quad C>0\;.
\end{equation}
We recall here a well-known estimate (see~\cite{5}):
\begin{lemma}\label{ots-sv} We have\begin{equation}\label{8}
M_{n,\delta}\le\sqrt{\frac{n+1}\delta}\;,\quad n{\in}\mathbb N\;.
\end{equation}\end{lemma}
\begin{proof}Indeed, this is immediate from the estimate
\eqref{chezaro}. We can also argue differently:
\[
1=\int_{\mathbb T}|\phi_n|^2d\sigma\ge\delta/(2\pi)\int_{\mathbb T}|\phi_n|^2d\theta
\]
so \eqref{8} follows from
\[
\|\phi_n\|_{L^2(\mathbb{T})}^2=\left\|\vphantom\sum\smash{\sum^n_{j=0}}c_jz^j\right\|_{L^2(\mathbb{T})}^2=
2\pi\sum^n_{j=0}|c_j|^2\le\frac{2\pi}\delta
\]
and Cauchy-Schwarz
\[
\|\phi_n\|_{L^\infty(\mathbb{T})}\le (n+1)^{1/2}\sqrt{\sum_{j=0}^n
|c_j|^2}\, .
\]
\end{proof}
{\bf Remark.} Notice that all we used in the proof is the
normalization $\|\phi_n\|_{L^2(\mathbb{T},\sigma)}=1$ and the
Steklov's condition on the measure. The problem, though, is whether
the orthogonality leads to further restrictions on the size.

\bigskip

The purpose of the current paper is to obtain the sharp bounds for
the problem of Steklov, i.e., the problem of estimating the growth
of $\phi_n$. We will get rid of the logarithmic factor in the
denominator in \eqref{est-ra} and \eqref{7} and thus prove the
optimal inequalities. The main results are contained in the
following two statements:\bigskip
\begin{theorem}\label{T3-i}If $\delta\in (0,1)$, then
\begin{equation}\label{osnova-i}M_{n,\delta} > C(\delta)\sqrt n\,\, .
\end{equation}
\end{theorem}
\smallskip and\smallskip
\begin{theorem}\label{rrra-i}Let $\delta{\in}(0,1)$. Then, for every positive sequence
$\{\beta_n\}:\lim_{n\to\infty}\beta_n=0$, there is a probability
measure
$\sigma^*:d\sigma^*={\sigma^*}'d\theta,\quad\sigma^*{\in}S_\delta$
such that
\begin{equation}\label{est-ra-i}
\|\phi_{k_n}(z,\sigma^*)\|_{L^\infty(\mathbb T)}\ge
\beta_{k_n}\sqrt{k_n}
\end{equation}
for some sequence $\{k_n\}\subset\mathbb N$.\end{theorem}

{\bf Remark.} It will be clear later that both results hold for far
more regular weights (see Lemma~\ref{build-up} and the proof of
Theorem~\ref{rrra-i} below).

\bigskip

The Steklov condition \eqref{gabor} is quite natural for the
analysis of the ``size" of the polynomial. Indeed, if $\mu$ is a
measure (not necessarily a probability one) then the trivial scaling
\[
\phi_n(z,m\mu)=\frac{\phi_n(z,\mu)}{\sqrt m}
\]
holds and this changes the size of $\phi_n$ accordingly. Let now
$\mu$ be a probability measure and $\Gamma$ be a small arc which
does not support  all of $\mu$. Then, we can can take
$\mu_m=\chi_{\Gamma^c}\cdot\mu+m\chi_\Gamma\cdot\mu$ with $m$ very
small. So, $\|\mu_m\|\sim\|\chi_{\Gamma^c}\cdot\mu\|\sim 1$ and one
can expect that $\phi_n(z,m\mu)$ gets large on most of $\Gamma$ as
$m\,{\to}_{}0$ in analogy to the case of the whole $\mathbb T$. This
is indeed true for many measures $\mu$. Thus, if one studies the
dependence of $\|\phi_n(z,\mu)\|_{L^\infty(\mathbb T)}$ on $\mu$,
then the conditions on the measure which control the size of the
polynomial should account for that fact and an obvious bound that
takes care of this is \eqref{gabor} as it does not allow the measure
to be scaled on any arc.\smallskip

The problem of estimating the size of $\phi_n$ is one of the most
basic and most well-studied problems in approximation theory.
Nevertheless, the sharp bounds were missing even for the Steklov's
class -- the most natural class of measures for this problem. In the
current paper, we not only establish these bounds but  also suggest
a new method which, we believe, is general enough to be used in the
study of other variational problems where the constructive
information on the weight is given. For example, one can replace the
Steklov's condition by the lower bounds like
\[
\sigma'(\theta)\ge w(\theta), \quad {\rm for\, a.e.}\, \theta\in
\mathbb{T}
\]
where $w(\theta)$ vanishes at a point in a particular way (e.g.,
$w(\theta)\sim |\theta|^\alpha$). The results we obtain are sharp
and we apply them to estimate  polynomial entropies -- another
important quantity to measure the size of the polynomial.\smallskip

{\bf Remark.} The size, asymptotics, and universality of the
Christoffel-Darboux kernel were extensively studied, see, e.g.
\cite{doron1, doron2, mnt, tot, tot2}. We, however, will focus on
$\phi_n$ itself.
\bigskip

{\bf Remark.} Since $S_\delta$ is invariant under the rotation and
$\{e^{i\,j\theta_0}\phi_j(ze^{-i\theta_0},\mu)\}$ are orthonormal
with respect to $\mu(\theta-\theta_0)$, we can always assume that
$\|\phi_n\|_\infty$ is reached at point $z=1$. Therefore, we have
\[
M_{n,\delta}=\sup_{\mu\in S_\delta}|\phi_n(1,\mu)|\,\,.
\]
One can consider the monic orthogonal polynomials $\Phi_n(z,\mu)=z^n+\ldots$ and the Schur parameters
$\{\gamma_n\}$ so that
\[
\phi_n(z,\mu)=\frac{\Phi_n(z,\mu)}{\|\Phi_n\|_{L^2(\mathbb{T},\mu)}}
\]
and
\begin{equation}\label{schur}
\gamma_n=-\overline{\Phi_{n+1}(0,\mu)}
\end{equation}
If $\rho_n=\sqrt{1-|\gamma_n|^{2^{\vphantom+}}}$, then (see
\cite{sim1})
\begin{equation}\label{gabor1}
\Phi_n(z,\mu)=\phi_n(z,\mu)\lambda_n^{-1},\quad
\lambda_n=\Bigl(\rho_0\cdot\ldots\cdot\rho_{n-1}\Bigr)^{-1}\,\,.
\end{equation}
 The Szeg\H{o} formula \cite{sim1} yields
\begin{equation}\label{gabor2}
\exp\left(\frac1{4\pi}\int^\pi_{-\pi}\ln(2\pi\mu'(\theta))d\theta\right)=\prod_{j\,\ge0}\rho_j\,\,.
\end{equation}
So, for $\mu\in S_\delta$, we have
\[
\sqrt{\delta\,}\le\prod_{n\ge0}\rho_n\le1
\]
and therefore
\[
\sqrt{\delta\,}|\phi_n(z,\mu)|\le|\Phi_n(z,\mu)|\le|\phi_n(z,\mu)|,\quad z\in\mathbb C
\]
for any $\mu\in S_\delta$. Thus, we have
\begin{equation}\label{tot-por}
\delta^{1/2}M_{n,\delta}\le \sup_{\mu\in S_\delta}
\|\Phi_n(z,\mu)\|_{L^\infty(\mathbb T)} \le M_{n,\delta}
\end{equation}
and for fixed $\delta$ the variational problems for
orthonormal and monic orthogonal polynomials are equivalent.\bigskip

The estimate \eqref{8} can not possibly be sharp for $\delta$ very
close to $1$. Indeed, if $\delta=1$ then $\sigma$ is the Lebesgue
measure and $\phi_n(z)=z^n$. We have the following result which
provides an effective bound and improves \eqref{8} for $\delta$
close to $1$.
\begin{lemma}\label{l1}
We have
\[
M_{n,\delta}\le \delta^{-1/2}\left(
1+\sqrt{\frac{n(1-\delta)}{\delta}}\right)
\]
\end{lemma}
\begin{proof}
 Let
$\widetilde\sigma$ be one of the maximizers for $M_{n,\delta}$,
i.e., $
\|\phi_n(z,\widetilde\sigma)\|_{L^\infty(\mathbb{T})}=M_{n,\delta}
$. The existence of such a maximizer is proved in Theorem \ref{etre}
below. Then, $ d\widetilde\sigma=(2\pi)^{-1}\delta
d\theta+d\widetilde\mu $ where $\|\widetilde\mu\|=1-\delta$. Let
$\Phi_n(z,\widetilde \sigma)$ be the corresponding monic polynomial.
We use the variational characterization of $\Phi_n$ (see \cite{6}):
$ \Phi_n=\arg \min_{P(z)=z^n+\ldots}
\|P\|^2_{L^2(\mathbb{T},\widetilde\sigma)} $. If
$\Phi_n=z^n+\widetilde a_{n-1}z^{n-1}+\ldots+\widetilde
a_1z+\widetilde a_0$, then
\[
\widetilde a=\arg D_{n,\delta},\, D_{n,\delta}=\min_{a=(a_0,
\ldots,\, a_{n-1})} \left(\delta\|a\|^2_{\ell^2}+\int_{\mathbb{T}}
|z^n+ a_{n-1}z^{n-1}+\ldots+ a_1z+ a_0|^2d\widetilde\mu\right)
\]
In particular, upon choosing $a=0$, we get $ D_{n,\delta}\le
\int_{\mathbb{T}} d\widetilde\mu=1-\delta $ and so $ \|\widetilde
a\|^2_{\ell^2}\le (1-\delta)/\delta $. Then, Cauchy-Schwarz
inequality gives
\[
\|\Phi_n\|_{L^\infty(\mathbb{T})}\le
1+\sqrt{\frac{n(1-\delta)}{\delta}}
\]
and \eqref{tot-por} finishes the proof.
\end{proof}

Now, we would like to comment a little on the methods we use. The
proofs by Rakhmanov were based on the following formula for the
orthogonal polynomial that one gets after adding several point
masses to a ``background" measure at particular locations on the
circle (see \cite{3}).

\begin{lemma}Let $\mu$ be a positive measure on $\mathbb T$, $\Phi_n(z,\mu)$ be the corresponding
monic orthogonal polynomials, and
\[
K_n(\xi,z,\mu)=\sum^n_{l=0}\overline{\phi_j(\xi,\mu)}\phi_j(z,\mu)
\]
be the Christoffel-Darboux kernel, i.e.
\[
P(\xi)=\langle
P(z),K_n(\xi,z,\mu)\rangle_{L^2(\mathbb{T},\mu)},\quad\deg P\le
n\,\,.
\]
Then, if $\xi_j\in\mathbb T,j=1,\,...\,m,\,m\le n$ are chosen such that
\begin{equation}\label{uslo-k}
K_{n-1}(\xi_j,\xi_l,\mu)=0, \quad j\neq l\end{equation}
then
\begin{equation}
\Phi_n(z,\eta)=\Phi_n(z,\mu)-\sum_{k=1}^m
\frac{m_k\Phi_n(\xi_k,\mu)}{1+m_kK_{n-1}(\xi_k,\xi_k,\mu)}K_{n-1}(\xi_k,z,\mu)
\end{equation}
where
\[
\eta=\mu+\sum_{k=1}^m m_k\delta_{\theta_k}, \quad z_k=e^{i\theta_k},\quad m_k\ge 0\,\,.
\]
\end{lemma}

The limitation that $\xi_j$ must be the roots of $K$ is quite
restrictive and the direct application of this formula with
background $d\mu=d\theta$ yields logarithmic growth at best. In the
later paper \cite{4}, Rakhmanov again  ingeniously used the idea of
inserting the point mass but the resulting bound (\ref{7}) contained
the logarithm in the denominator and the measure of orthogonality
was not defined explicitly.\smallskip

We will use a completely different approach. First, we will rewrite
the Steklov condition in the convenient form as some estimate that
involves Caratheodory function and a polynomial (see Lemma
\ref{decop} below). This decoupling is basically equivalent to
solving the well-known truncated trigonometric moments problem.
Then, we will present a particular function and a polynomial and
show that they satisfy the necessary conditions. This allows us to
have a good control on the size of the polynomial itself and on the
structure of the measure of orthogonality.\vspace{1cm}

The paper has four parts and two Appendixes. The first part contains
results on the structure of an optimal measure and discussion of the
case when  $\delta$ is $n$-dependent and very small. In the second
part, the proof of Theorem \ref{T3-i} is given for fixed small
$\delta$.   We will apply the ``localization principle" to handle
every $\delta\in (0,1)$  and prove Theorem \ref{rrra-i} in the third
part. In the last one, two applications are given. First, the lower
bounds are obtained for polynomials orthogonal on the real line.
Then, we prove the sharp estimates for the polynomial entropies in
the Steklov class. The Appendixes contain some auxiliary results we
use in the main text. \vspace{1cm}

Here are some notation used in the paper: the Cauchy kernel for the unit circle is denoted by $C(z,\xi)$, i.e.
\[
C(z,\xi)=\frac{\xi+z}{\xi-z},\quad\xi\in\mathbb T\,\,.
\]
If the function is analytic in $\mathbb D$ and has a nonnegative
real part there, then we will call it Caratheodory function.

 Given
any polynomial $P_n(z)=p_nz^n+\ldots+p_1z+p_0$, we can define its
$n$-th reciprocal (or the $*$--transform)
\[
P_n^*(z)=z^n \overline{P_n(1/{\overline
z})}=\overline{p}_0z^n+\overline{p}_1z^{n-1}+\ldots+\overline{p}_n\,\,.
\]
Notice that if $z^*$ is a root of $P_n(z)$ and $z^*\neq 0$, then
$(\overline{z^*})^{-1}$ is a root of $P_n^*(z)$.

Given two positive functions $F_1$ and $F_2$ defined on $\mathbb D$,
we write $F_1\lesssim F_2$ if there is a constant $C$ (that might
depend only on the fixed parameters) such that
\[
F_1<CF_2~,\quad C>0
\]
on $\mathbb D$. We write $F_1\sim F_2$ if
\[
F_1\lesssim F_2\lesssim F_1\,\,.
\]
We use the notation $b=O^*(a)$ if $b\sim a$. The symbol $\delta_a$
denotes the delta function (the point mass) supported at $a\in
(-\pi,\pi]$ or at complex point $e^{ia}\!\in\!\mathbb T$. If $\e$ is
a positive parameter, then $\e\,{\ll}\,1$ is the shorthand for:
``$\e<\e_0$, where $\e_0$ is sufficiently small". If
$p(z)=a_nz^n+\ldots+a_1z+a_0$, then we define ${\rm
coeff}(p,j)=a_j$.

We will use the following standard notation for the norms. If $\mu$
is a measure, $\|\mu\|$ refers to its total variation. For functions
$f$ defined on $[-\pi,\pi]$, we write
\[
\|f\|_p=\|f\|_{L^p[-\pi,\pi]}, \quad 1\le p\le \infty\,.
\]
The symbol $\langle f, g\rangle_\sigma$ denotes the following  inner
product given by
\[
\langle f, g\rangle_\sigma=\int_{M} f(x)\overline g(x)d\sigma
\]
where $\sigma$ is a measure on $M$ (e.g., $M=\mathbb{T}$ or
$M=[-\pi,\pi]$).

\vspace{1.5cm} {\Large \part{ Variational problem:  structure of the
extremizers}}\bigskip
\section{Structure of the extremal measure.}

In this section, we first address the problem of the existence of maximizers, i.e., $\mu^*_n\in S_\delta$
for which\begin{equation}\label{extre}
M_{n,\delta}=|\phi_n(1;\mu^*_n)|\,\,.
\end{equation}
We will prove that these extremizers exist and will study their properties.

\begin{theorem}\label{etre}
There are $\mu^*_n\in S_\delta$ for which \eqref{extre} holds.
\end{theorem}\begin{proof}
Suppose $\mu_k\in S_\delta$ is the sequence which yields the $\sup$, i.e.
\[
|\phi_n(1,\mu_k)|\to M_{n,\delta}, \quad k\to\infty\,\,.
\]
Since the unit ball is weak-($\ast$) compact, we can choose
$\mu_{k_j}\to \mu^*$ and this convergence is weak-($\ast$), i.e.
\[
\int fd\mu_{k_j}\to \int fd\mu^*, \quad j\to\infty
\]
for any $f\in C(\mathbb T)$. In particular, $\mu^*$ is a probability measure. Moreover,
for any interval $(a,b)\subseteq(-\pi,\pi]$, we have (assuming, e.g., that the endpoints $a$ and $b$
are not atoms for $\mu^*$):
\[
\int_{[a,b]} d\mu^*\ge \delta(b-a)/(2\pi)
\]
since each $\mu_{k_j}\in S_\delta$. This implies
${\mu^*}'\ge\delta/(2\pi)$ a.e. on $\mathbb T$. The moments of
$\mu_{k_j}$ will converge to the moments of $\mu^*$ and therefore
\[
\phi_n(1,\mu_{k_j})\to\phi_n(1,\mu^*)\,\,.
\]
Therefore, $\mu^*\in S_\delta$ and $|\phi_n(1,\mu^*)|=M_{n,\delta}$.
\end{proof}

This argument gives existence of an extremizer. Although we do not know whether it is unique,
we can prove that every $d\mu^*$ must have a very special form.

\begin{theorem}\label{spec-forma}
If $\mu^*$ is a maximizer then it can be written in the following
form
\begin{equation}\label{form}
d\mu^*=(2\pi)^{-1}\delta\,d\theta+\sum^N_{j=1}m_j\delta_{\theta_j},\quad
1\le N\le n
\end{equation}
where $m_j\ge 0$ and $-\pi<\theta_1<\ldots<\theta_N\le\pi$.
\end{theorem}

Suppose we have a positive measure $\mu$ and its moments are given
by
\[
s_j=\int e^{ij\theta}d\mu=s_j^R+i\;\!s_j^I,\quad j=0,1,\ldots
\]
Then, the following formulas are well-known (\cite{sim1})
\begin{equation}\label{det1}
\Phi_n(z,\sigma)={\frak D}_{n-1}^{-1}\left|\begin{array}{cccc}
s_0 & s_1 & \ldots& s_n\\
\overline{s}_{1} & s_0 & \ldots& s_{n-1}\\
\ldots& \ldots& \ldots& \ldots\\
\overline{s}_{n-1} & \overline{s}_{n-2} &\ldots&s_1 \\
1 & z& \ldots& z^n\end{array}\right|\,
\end{equation}

\begin{equation}\label{det2}
{\frak D}_n=\det {\frak T}_n,\quad {\frak
T}_n=\left[\begin{array}{cccc}
s_0 & s_1 & \ldots& s_n\\
\overline{s}_{1} & s_0 & \ldots& s_{n-1}\\
\ldots& \ldots& \ldots& \ldots\\
\overline{s}_{n-1} & \overline{s}_{n-2} &\ldots&s_1 \\
\overline{s}_{n} & \overline{s}_{n-1}& \ldots&
s_0\end{array}\right]\,\,.
\end{equation}
These identities show that $\Phi_n(z,\sigma)$ depends only on the
first $n$ moments of the measure $\sigma$:
$\Phi_n(z,\sigma)=\Phi_n(z,s_0,\ldots,s_n)$.  Moreover, by
definition of the monic orthogonal polynomial,
\[
\Phi_n(z,s_0,s_1,\ldots, s_n)=\Phi_n(z,1, s_1/s_0, \ldots,
s_n/s_0)\, ,
\]
i.e., $\Phi_n(z,\sigma)$ does not depend on the normalization of the
measure.

The functions $F_{1(2)}$, given by
\[
F_1(s)=|\Phi_n(1,s_0,\ldots,s_n)|^2,\quad
F_2(s)=|\phi_n(1,s_0,\ldots,s_n)|^2, \quad
s=(s_0,s_1^R,\ldots,s_n^I)\in \mathbb{R}^{2n+1}
\]
are the smooth functions of the variables
$\{s_0,s_j^R,s_j^I\},j=1,\ldots,n$ wherever they are defined.
Consider $\Omega_n=\{s:{\frak T}_n(s)>0\}$. If $s\in\Omega_n$, then
there is a family of measures $\mu$ which have $(s_0,\ldots,s_n)$ as
the first $n$ moments. That follows from the solution to the
truncated trigonometric moment problem.

We will need the following
\begin{lemma}
The functions $F_{1(2)}(s)$ do not have stationary points on
$\Omega_n$.
\end{lemma}
\begin{proof}
It is known \cite{sim1} that the map between the first $n$ Schur
parameters (see \eqref{schur}) and the first $n$ moments of a
probability measure, i.e., $\{\gamma_j\}|_{j=0}^{n-1}\in
\mathbb{D}^n\to (1,s_1,\ldots, s_n)$, is a bijection.  The formulas
\eqref{schur} and \eqref{det1} imply that
\begin{equation}\label{bie}
\gamma_{n-1}=\overline{s}_n\frac{\frak{D}_{n-2}}{\frak{D}_{n-1}}+f(s_0,s_1,\overline
s_1, \ldots, s_{n-1}, \overline s_{n-1}) \,.
\end{equation}
The both polynomials $\Phi_n$ and $\phi_n$ satisfy the recurrences
(\cite{sim1}, p.57)
\[
\Phi_n(1,\sigma)=\Phi_{n-1}(1,\sigma)-\overline\gamma_{n-1}\Phi_{n-1}^*(1,\sigma)
\]
and
\[
\phi_n(1,\sigma)=\rho_{n-1}^{-1}(\phi_{n-1}(1,\sigma)-\overline\gamma_{n-1}\phi_{n-1}^*(1,\sigma))\,,
\]
which shows that
\[
\nabla_{s_n}{ |\Phi_{n}(1,\sigma)|^2}=\left( \frac{\partial
|\Phi_n(1,\sigma)|^2}{\partial s_n^R},\frac{\partial
|\Phi_n(1,\sigma)|^2}{\partial s_n^I} \right)\neq 0\,,
\]
because $ \Phi_{n-1}(z,\sigma)$ and $ \Phi_{n-1}^*(z,\sigma)$ do not
depend on $s_n$, $\Phi_{n-1}^*(1,\sigma)\neq 0$, and
$\Phi_{n}(1,\sigma)\neq 0$. Since
$\rho_{n-1}=\sqrt{1-|\gamma_{n-1}|^2}$, we have
\[
|\phi_n(1,\sigma)|^2=C\frac{|\xi-\gamma_{n-1}|^2}{1-|\gamma_{n-1}|^2},
\quad
\xi=\frac{\overline\phi_{n-1}(1,\sigma)}{\overline\phi_{n-1}^*(1,\sigma)},
\quad |\xi|=1, \quad C=|\phi_{n-1}^*(1,\sigma)|^2\,.
\]
We can rewrite it as
\[
|\phi_n(1,\sigma)|^2=C\frac{1+|\gamma_{n-1}|^2-2\Re
(\gamma_{n-1}\overline \xi)}{1-|\gamma_{n-1}|^2}\,,
\]
which shows that
\[
\nabla_{\gamma_{n-1}}|\phi_n(1,\sigma)|^2\neq 0\,.
\]
where $|\phi_n(1,\sigma)|^2$ is considered as a function of
$\{\gamma_0,\ldots,\gamma_{n-1}\}$.

 Now \eqref{bie} yields
\[
\nabla_{s_{n}}|\phi_n(1,\sigma)|^2\neq 0
\]
and the proof is finished.
\end{proof}{\bf Remark.}
The proof actually shows that ${\nabla_{s_n} F_{1(2)}}\neq 0$.

\begin{proof}{\it (of the Theorem \ref{spec-forma})}
Our variational problem is an extremal problem for a functional
$F(s_0,s_1^R,\ldots,s_n^I)$ on the finite number of moments
$\{s_0,s_1^R,\ldots,s_n^I\}$ of a measure $\mu$ from $S_\delta$. We
can take
$$
F(s_0,s_1^R,\ldots,s_n^I)=|\phi_n(1,s_0,\ldots,s_n)|^2\;.
$$
The function $F$ is differentiable. Moreover,
\[
s_0=\int d\mu,\quad s_j^R=\int\cos(j\theta)d\mu,\quad s_j^I=\int\sin(j\theta)d\mu\,\,.
\]
Considering the moments as functionals of $\mu$, we compute the
derivative of $F$ at the point $\mu^*$ in the direction $\delta\mu$:
\[
dF=\int\left(\frac{\partial F}{\partial s_0}(s^*)+\frac{\partial F}{\partial
s_1^R}(s^*)\cos(\theta)+\ldots+\frac{\partial F}{\partial s_n^I}(s^*)\sin(n\theta)\right)d(\delta\mu)\,\,.
\]
Consider the trigonometric polynomial:
\[
T_n(\theta)=\frac{\partial F}{\partial s_0}(s^*)+\frac{\partial F}{\partial
s_1^R}(s^*)\cos(\theta)+\ldots+\frac{\partial F}{\partial s_n^I}(s^*)\sin(n\theta)\,\,.
\]
From the previous Lemma and Remark, we know that it has degree $n$.
Let $M=\max T_n(\theta)$ and $\{\theta_j; j=1,\ldots,N\}$ are all
points where $M$ is achieved. Clearly, $N\le n$.

Now, if we find a smooth curve  $\mu(t),~t\in(0,1]$ such that $\mu(t)\in S_\delta, \mu(1)=\mu^*$ and define
\[
H(t)=F(s_0(\mu(t)),s_1^R(\mu(t)),\ldots,s_n^I(\mu(t)),
\]
then $H'(1)\ge 0$ as follows from the optimality of $\mu^*$.

Now, we will assume that the measure $\mu^*$ is not of the form
\eqref{form} and then will come to a contradiction by choosing the
curve $\mu(t)$ in a suitable way.

We will first prove that the singular part of $\mu^*$ can be
supported only at the points $\{\theta_j\}$. Indeed, suppose we have
$$\mu^*=\mu_1+\mu_2$$
where $\mu_2$ is singular and supported away from $\{\theta_j\}$.
Consider two smooth functions $p_1(t)$ and $p_2(t)$ defined on
$(0,1]$ that satisfy
\[
\|\mu_1\|+p_1(t)+p_2(t)\|\mu_2\|=1,\quad p_{1(2)}(t)\ge 0,\quad
p_1(1)=0,\quad p_2(1)=1\,\,.
\]
For example, one can take $p_1(t)=\|\mu_2\|(1-t),\, p_2(t)=t$. Take
$\mu(t)=\mu_1+p_1(t)\delta_{\theta_1}+p_2(t)\mu_2$. We have
$\mu(t)\in S_\delta$ and
\[
H'(1)=\int T_n(\theta)d\mu_2-\|\mu_2\|T_n(\theta_1)<0
\]
since $\theta_1$ is a point of global maximum for $T_n$ and $\mu_2$
is supported away from $\{\theta_j\}$ by assumption. This
contradicts optimality of $\mu^*$ and so $\mu_2=0$.

We can prove similarly now that $(\mu^*)'=(2\pi)^{-1}\delta$ a.e.
Indeed, suppose
\[
\mu^*=\mu_1+\mu_2,\quad
d\mu_2=f(\theta){{}^{{}_{\textstyle\chi}}}_{\Omega}d\theta
\]
where $f(\theta)>(2\pi)^{-1}\delta_1>(2\pi)^{-1}\delta$ on $\Omega$,
$|\Omega|>0$ and $\mu_1$ is supported on $\Omega^c$. We consider the
curve
\[
\mu(t)=\mu_1+p_1(t)\delta_{\theta_1}+p_2(t)\mu_2(t)\,\,.
\]
The choice of $p_{1(2)}$ is the same. Then, $\mu(t)\!\in\!S_\delta$
for $t\!\in\!(1\!-\!\e,1)$ provided that $\e(\delta_1)$ is small.
The similar calculation yields $H'(1)<0$ and that gives a
contradiction.
\end{proof}

Since the maximizer in the Steklov problem is given by \eqref{form},
we want to make an observation. The following result is attributed
to Geronimus (see \cite{5}).

\begin{lemma}
Consider $\mu(t)=(1-t)\mu+t\delta_\beta$ where $t\in (0,1), \beta\in
[-\pi,\pi)$. Then,
\begin{equation}\label{insert}
\Phi_n(z,\mu(t))=\Phi_n(z,\mu)-t\frac{\Phi_n(\xi,\mu)K_{n-1}(\xi,z,\mu)}{1-t+tK_{n-1}(\xi,\xi,\mu)}\,,\xi=e^{i\beta}\,.
\end{equation}\end{lemma}
\begin{proof}
Notice that the right hand side is a monic polynomial of degree $n$.
Then,
\[
\langle {\rm r.h.s.},z^j\rangle_{\mu(t)}=0, \quad j=0,\ldots, n-1
\]
which yields orthogonality.
\end{proof}

The formula \eqref{insert} expresses monic polynomials obtained by
adding one point mass to an arbitrary measure at any location. One
can try to iterate it to get the optimal measure $d\mu^*$. That,
however, leads to very complicated analysis. \bigskip

\section{The regime of small $n$--dependent $\delta$.}

One can make a trivial observation that if $\mu$ is any positive measure (not necessarily a probability one)
and $\phi_n(z,\mu)$ is the corresponding orthonormal polynomial, then
\begin{equation}\label{scaling}
\phi_n(z,\alpha\mu)=\alpha^{-1/2}\phi_n(z,\mu)
\end{equation}
for every $\alpha>0$. The monic orthogonal polynomials, though, stay unchanged
\[
\Phi_n(z,\alpha\mu)=\Phi_n(z,\mu)\,\,.
\]\bigskip
Now, consider the modification of the problem: we define
\[
\~M_{n,\delta}=\sup_{\mu'\ge\,\delta/(2\pi)}\|\phi_n(z,\mu)\|_{L^\infty(\mathbb{T})}=
\sup_{\mu'\ge\,\delta/(2\pi)}|\phi_n(1,\mu)|,
\]
i.e., we drop the requirement for the measure $\mu$ to be a
probability measure. In this case, the upper estimate for
$\~M_{n,\delta}$ stays the same and the proof of
\[
\~M_{n,\delta}\le\sqrt{\frac{n\!+\!1}\delta}\,\,.
\]
is identical. It turns out that the sharp lower bound  can be easily
obtained in this case.
\begin{theorem}We have
\[
\~M_{n,\delta}=\sqrt{\frac{n\!+\!1}\delta}\,\,.
\]\end{theorem}
\begin{proof}Consider\begin{equation}\label{19}
d\sigma=\frac\delta{2\pi}\,d\theta+\sum^n_{k=1}m_k\,\delta_{\theta_k}\;,\quad\theta_k=\frac
k{n\!+\!1}2\pi\,,\quad k=1,\ldots,n\,.
\end{equation}
We assume that all $m_k\ge 0$. Consider
\begin{equation}\label{21}
\Pi_n(z)=\prod\limits^n_{k=1}(z-\e_k)\;,\quad\e_k=e^{i\theta_k}.
\end{equation}
One gets:
$\Pi_n(z)=1+z+\ldots+z^n,\quad\|\Pi_n\|^2_{L^2(\mathbb{T},\sigma)}=\delta(n\!+\!1).$
We define now
\[
\Phi_n=\Pi_n+Q_{n-1}
\]
where $Q_{n-1}(z)=q_{n-1}z^{n-1}+\ldots+q_1z+q_0$ is chosen to
guarantee the orthogonality
$\langle\Phi_n,z^j\rangle_{\sigma}=0,~j=0,\ldots,n-1$. Suppose that
$m_k=m$ for all $k$. Then, we have the following equations
\[
\delta+\delta q_j+m\sum_{l=0}^{n-1}q_l\sum^n_{k=1}\e_k^{l-j}=0,\quad
j=0,\ldots,n-1\,\,.
\]Then, since\[
\sum^n_{k=0}\e_k^d=0,\quad d\!\in\!\{-n,\ldots,-1,1,\ldots,n\}
\]we get\[
\delta+\delta q_j+m(n+1)q_j -m\sum_{l=0}^{n-1}q_l=0,\quad
j=0,\ldots n{-}1
\]and\[
q_j=-\frac\delta{\delta+m},\quad j=0,\ldots n{-}1\,\,.
\]Thus,\[
\Phi_n(z)=\Phi_n(z,\sigma)=1+\ldots+z^n-\frac{\delta}{\delta+m}(1+\ldots+z^{n-1})=\frac
m{\delta+m}\Pi_n(z)+\frac\delta{\delta+m}z^n\,\,.
\]Now, we have\[
\|\Phi_n\|_{L^\infty(\mathbb{T})}=\Phi_n(1,\sigma)=n+1-\frac\delta{\delta+m}n=1+\frac{mn}{\delta+m}
\]and\[
\|\Phi_n\|^2_{L^2(\mathbb{T},\sigma)}=\delta\left(1+\frac{m^2n}{(\delta+m)^2}\right)+\frac{\delta^2
nm}{(\delta+m)^2}= \delta\left(1+\frac{mn}{\delta+m}\right)\,\,.
\]
For the orthonormal polynomial,
\[
\phi_n(1,\sigma)=\frac{\Phi_n(1,\sigma)}{\|\Phi_n\|_{L^2(\mathbb{T},\sigma)}}\,\,.
\]
For fixed $n$, this gives
\[
\lim_{m\to\infty}\|\phi_n\|_{L^\infty(\mathbb{T})}=\sqrt{(n+1)/\delta}\,\,.
\]\end{proof}
{\bf Remark.} This Theorem has the following implication for our
original problem. Suppose we consider the class $S_\delta$ but
$\delta$ is small in $n$. Then, \eqref{scaling} gives
\[
M_{n,\delta_n}=\sqrt{\frac{n+1}{\delta_n}}(1+o(1))
\]
where
\[
\delta_n=\frac{C}{nm_n},\quad m_n\to+\infty,\quad{\rm as}\quad n\to\infty\,\,.
\]
Thus, for $\delta$ small in $n$, the upper bound for $M_{n,\delta}$
is sharp. If one takes $m_n=1/n$ in the proof above to make the
total mass finite, the polynomials $\phi_n$  are bounded in $n$ as
$\delta\sim1$.\bigskip

\vspace{1.5cm} {\Large \part{ The proof of Theorem \ref{T3-i}: the
case of  small fixed $\delta$}}\bigskip

\section{Lower bounds: fixed $\delta$ and $n$.}\label{sec.3}

In this section, we prove the sharp lower bound for small fixed
$\delta$. The main result is the following Theorem.

\begin{theorem}\label{T3}  There is $\delta_0\in (0,1)$ such that
\begin{equation}\label{osnova}
M_{n,\delta_0} \gtrsim \sqrt n\,\,.
\end{equation}\end{theorem}

{\bf Remark.} In this section, we are not trying to control the size
of $\delta_0$. The full range $\delta\in (0,1)$ will be covered in
part 3 by using certain localization technique.

\subsection{Notation and Basics from the theory of polynomials orthogonal on the
circle.}

 We start by introducing some notation and recalling the relevant
facts from the theory of polynomials orthogonal on the unit circle.

The following trivial Lemma will be needed later (see, e.g.,
\cite{PS}, p. 108)
\begin{lemma}\label{noli}
Let $n\ge1$. If a polynomial $P_n$ of degree at most $n$ has all
zeroes outside $\overline{\mathbb D}$, then $D_n(z)=P_n(z)+P_n^*(z)$
has all (exactly n) zeroes on the unit circle.
\end{lemma}
\begin{proof}We have
\[
D_n(z)=P_n\left(1+\frac{P_n^*}{P_n}\right),\quad z\!\in\!\mathbb D\,\,.
\]
The first factor has no zeroes in $\mathbb D$. In the second one,
$P_n^*/P_n$ is a Blaschke product (indeed,  $|P_n^*/P_n|\!=\!1$ on
$\mathbb{T}$). So, $\displaystyle 1+\frac{P_n^*}{P_n}$ is
holomorphic in $\mathbb D$, continuous up to the boundary, and its
boundary values belong to the circle with center at $z=1$ and radius
$1$. Thus,
$$\Re\Big(1+\frac{P_n^*}{P_n}\Big)\ge0\;.$$ Therefore,$$\text{either
}\quad\Re\Big(1+\frac{P_n^*(z)}{P_n(z)}\Big)>0~(z\!\in\!\mathbb
D)\quad\text{ or }\quad1+\frac{P_n^*(z)}{P_n(z)}\equiv0.$$ The last
condition implies $P_n=-P^*_n$, so $P_n\ne0$ on $\overline{\mathbb
D}$ by assumption of the Lemma and $P_n\ne0$ on $\overline{\mathbb
C}\setminus\mathbb D$ because it is equal to $-P^*_n$. Then,
$P_n\ne0$ on $\overline{\mathbb C}$ and so
$P_n\equiv\operatorname{const}$. Therefore
$P^*_n=-P_n\equiv\operatorname{const}$. This is possible for $n=0$
only. Finally,
\[
1+\frac{P_n^*}{P_n}
\]
does not have zeroes in $\mathbb D$. Since $D_n$ is invariant under
the $*$--transform, it has the following property: $D_n(w)=0$
implies $D_n(\overline{w}^{-1})=0$. Therefore, $D_n$ has no zeroes
in $|z|>1$ as well.

 One can actually show that $D_n$ has the degree $n$ under the assumptions of the Lemma.
Indeed, if $D_n(z)\!\ne\!0$, $z\!\in\!\mathbb D$, then ${\rm
Var}\arg D_n\big|_{r\mathbb T}\!=\!0$, $r\!<\!1$. Therefore $\;{\rm
Var}\arg D^*_n\big|_{\frac1r\mathbb T}\!=\!2\pi n$
$\implies\deg(D^*_n)\!\ge\!n$. But $D^*_n=D_n$.\bigskip
\end{proof}

 We will be mostly working with the orthonormal polynomials
$\phi_n$ and the corresponding $\phi_n^*$. It is well known
\cite{sim1} that all zeroes of $\phi_n$ are inside $\mathbb D$ thus
$\phi_n^*$ has no zeroes in $\overline{\mathbb D}$. However, we also
need to introduce the second kind polynomials $\psi_n$ along with
the corresponding $\psi_n^*$. Let us recall (\cite{sim1}, p. 57)
that
\begin{equation}\hskip9em\left\{\begin{array}{cc}
\phi_{n+1}=\rho_n^{-1}(z\phi_n-\overline\gamma_n\phi_n^*),&\phi_0=\sqrt{1/|\mu|}=1\\
\phi_{n+1}^*=\rho_n^{-1}(\phi_n^*-\gamma_n z\phi_n),&\phi^*_0=\sqrt{1/|\mu|}=1
\end{array}\right.\quad\text{(probability case)}
\label{srecurs}
\end{equation}
and the second kind polynomials satisfy the recursion with Schur
parameters $-\gamma_n$, i.e.,
\begin{equation}\label{secon}\left\{\begin{array}{cc}
\psi_{n+1}=\rho_n^{-1}(z\psi_n+\overline\gamma_n\psi_n^*),&\psi_0=\sqrt{|\mu|}=1\\
\psi_{n+1}^*=\rho_n^{-1}(\psi_n^*+\gamma_n z\psi_n),&\psi^*_0=\sqrt{|\mu|}=1
\end{array}\right.\end{equation}
The following Bernstein-Szeg\H{o} approximation result is valid:
\begin{theorem}Suppose $d\mu$ is a probability measure and $\{\phi_j\}$ and $\{\psi_j\}$ are the
corresponding orthonormal polynomials of the first/second kind, respectively. Then, for any $N$, the function
\[
F_N(z)=\frac{\psi_N^*(z)}{\phi_N^*(z)}=\int_{\mathbb T}
C(z,e^{i\theta})d\mu_N(\theta),\quad d\mu_N(\theta)=\frac{d\theta}{2\pi|\phi_N(e^{i\theta})|^2}=
\frac{d\theta}{2\pi|\phi^*_N(e^{i\theta})|^2}
\]
has the first $N$ Taylor coefficients  identical to the Taylor
coefficients of the function
\[
F(z)=\int_{\mathbb T}C(z,e^{i\theta})d\mu(\theta)\,\,.
\]
In particular, the polynomials $\{\phi_j\}$ and $\{\psi_j\}$, $j\!\le\!N$ are the orthonormal polynomials
of the first/second kind for the measure $d\mu_N$.
\end{theorem}
We also need the following Lemma:
\begin{lemma}\label{vspomag}
The polynomial $P_n(z)$ of degree $n$ is the orthonormal polynomial
for a probability measure with infinitely many growth points if and
only if
\begin{itemize}
\item[1.] $P_n(z)$ has all $n$ zeroes inside $\mathbb D$ (counting the multiplicities).
\item[2.] The normalization conditions
\[
\int_\mathbb T\frac{d\theta}{2\pi|P_n(e^{i\theta})|^2}=1~,\quad\operatorname{coeff}(P_n,n)>0
\]
are satisfied.\end{itemize}\end{lemma} Now, we are ready to
formulate the main result of this section.

\subsection{The reduction of the problem:  Decoupling Lemma}

The proof of the Theorem \ref{T3} will be based on the following
result.

\begin{lemma}{\rm \bf (The Decoupling Lemma)}\label{decop}
To prove \eqref{osnova}, it is sufficient to find a polynomial
$\phi_n^*$ and a Caratheodory function $\~F$ which satisfy the
following properties:
\begin{itemize}
\item[1.] $\phi_n^*(z)$ has no roots in $\mathbb D$.
\item[2.] Normalization on the size and ``rotation\!"
\begin{equation}\label{norma}
\int_\mathbb T|\phi_n^*(z)|^{-2}d\theta =2\pi~,\quad\phi_n^*(0)>0\,\,.
\end{equation}
\item[3.] Large uniform norm, i.e.,
$$|\phi^*_n(1)|\sim\sqrt n\,\,.$$
\item[4.] $\~F\!\in\!C^\infty(\mathbb T)$, $\Re\~F>0$ on $\mathbb T$, and
\begin{equation}\label{norka}
\frac1{2\pi}\int_\mathbb T\Re\~F(e^{i\theta})d\theta=1\,\,.
\end{equation}
\item[5.] Moreover,
\begin{equation}\label{main}
|\phi^*_n(z)|+|\~ F(z)(\phi_n(z)-\phi_n^*(z))|<C_1(\delta)\Bigl(\Re\~F(z)\Bigr)^{1/2}
\end{equation}
uniformly in $z\!\in\!\mathbb T$.
\end{itemize}\end{lemma}
\begin{proof}
By Lemma \ref{vspomag}, the first two conditions guarantee that
$\phi_n(z)$ is an orthonormal polynomial of some probability
measure. It also determines the first $n$ Schur parameters:
$\gamma_0,\ldots,\gamma_{n-1}$. The third one gives the necessary
growth. Next, let us show that the fourth and  fifth conditions are
sufficient for the existence of a measure $\sigma\!\in\!S_\delta$
for which $\phi_n$ is the $n$--th orthonormal polynomial.

By the fourth condition, $\~F$ defines the probability measure
$\~\sigma$ which is purely absolutely continuous and has positive
smooth density $\~\sigma'$ given by
\begin{equation}\label{fefl}
\~\sigma'(\theta)=\frac{\Re\~F(e^{i\theta})}{2\pi}
\,\,.
\end{equation}
Denote its Schur parameters by $\{\~\gamma_j\},\,j=0,1,\ldots$ and
the orthonormal polynomials of the first and second kind by
$\{\~\phi_j\}$ and $\{\~\psi_j\},\,j=0,1,\ldots$, respectively.
Notice that the normalization condition for $\~\sigma$ implies
$\~\phi_0\!=\!\~\psi_0\!=\!1$. By Baxter's Theorem \cite{sim1} we
have $\~\gamma_j\in\ell^1$ (in fact, the decay is much stronger but
$\ell^1$ is enough for our purposes). Then, let us consider the
probability measure $\sigma$ which has the following Schur
parameters
\[
\gamma_0,\ldots,\gamma_{n-1},\~\gamma_0,\~\gamma_1,\ldots
\]
We will show that this measure satisfies the Steklov's condition. Denote
\begin{equation}\label{otoj}
\gamma_n=\~\gamma_0,\gamma_{n+1}=\~\gamma_1,\ldots
\end{equation}
The Baxter Theorem implies that $\sigma$ is purely a.c., $\sigma'$
belongs to  Wiener's class $W(\mathbb T)$, and $\sigma'$ is positive
on $\mathbb T$. The first $n$ orthonormal polynomials corresponding
to the measure $\sigma$ will be $\{\phi_j\},\,j=0,\ldots,n-1$. Let
us compute the polynomials $\phi_j$ and $\psi_j$, orthonormal with
respect to $\sigma$, for the indexes $j>n$. Since the second kind
polynomials correspond to the Schur parameters $\{-\gamma_j\}$ (see
\eqref{secon}), the recursion can be rewritten in the following
matrix form
\begin{equation}\label{m-ca}
\left(\begin{array}{cc}
\phi_{n+m} & \psi_{n+m}\\
\phi_{n+m}^* & -\psi_{n+m}^*
\end{array}\right)=\left(\begin{array}{cc}
{\cal A}_m & {\cal B}_m\\
{\cal C}_m & {\cal D}_m
\end{array}\right)\left(\begin{array}{cc}
\phi_{n} & \psi_{n}\\
\phi_{n}^* & -\psi_{n}^*
\end{array}\right)\end{equation}
where ${\cal A}_m, {\cal B}_m, {\cal C}_m, {\cal D}_m$ satisfy
\begin{eqnarray*}\left(\begin{array}{cc}
{\cal A}_0 & {\cal B}_0\\
{\cal C}_0 & {\cal D}_0
\end{array}\right)=\left(\begin{array}{cc}
1 & 0\\
0 & 1
\end{array}\right),\hspace{6cm}\\
\left(\begin{array}{cc}
{\cal A}_m & {\cal B}_m\\
{\cal C}_m & {\cal D}_m
\end{array}\right)=\frac1{\~\rho_0\cdot\ldots\cdot\~\rho_{m-1}}\left(\begin{array}{cc}
z & -\~\gamma_{m-1}\\
-z\~\gamma_{m-1} & 1
\end{array}\right)\cdot\ldots\cdot\left(\begin{array}{cc}
z & -\~\gamma_0\\
-z\~\gamma_0 & 1
\end{array}\right)\end{eqnarray*}
and thus depend only on $\gamma_n,\ldots,\gamma_{n+m-1}$ (i.e.,
$\~\gamma_0,\ldots,\~\gamma_{m-1}$ by \eqref{otoj}.) Moreover, we
have
\[\left(\begin{array}{cc}
\~\phi_m & \~\psi_m\\
\~\phi_m^* & -\~\psi^*_m
\end{array}\right)=\left(\begin{array}{cc}
{\cal A}_m & {\cal B}_m\\
{\cal C}_m & {\cal D}_m
\end{array}\right)\left(\begin{array}{cc}
1 & 1\\
1 & -1
\end{array}\right)
\,\,.
\]
Thus, ${\cal A}_m\!=\!(\~\phi_m\,{+\,\~\psi_m)/2,~{\cal
B}_m\!=\!(\~\phi_m\,-}\,\~\psi_m)/2,~ {\cal
C}_m\!=\!(\~\phi^*_m\,{-\,\~\psi^*_m)/2,~{\cal
D}_m\!=\!(\~\phi^*_m\,+}\,\~\psi^*_m)/2$ and their substitution into
\eqref{m-ca} yields
\begin{equation}\label{intert}
2\phi_{n+m}^*=\phi_n(\~\phi_m^*-\~\psi^*_m)+\phi_n^*(\~\phi_m^*+\~\psi^*_m)=
\~\phi_m^*\left(\phi_n+\phi_n^*+\~
F_m(\phi_n^*-\phi_n)\right)\end{equation}
where
\[
\~F_m(z)=\frac{\~\psi^*_m(z)}{\~\phi^*_m(z)}\,\,.
\]
Since $\{\~\gamma_n\}\!\in\!\ell^1$ and $\{\gamma_n\}\!\in\!\ell^1$, we have (\cite{sim1}, p.~225)
\[
\~F_m\to\~F~{\rm as~}m\to\infty~{\rm and~}
\phi_n^*\to\Pi,~\~\phi_n^*\to\~\Pi~{\rm as~}n\to\infty\,\,.
\]
uniformly on $\overline{\mathbb D}$. The functions $\Pi$ and $\~\Pi$
are the Szeg\H{o} functions of $\sigma$ and $\~\sigma$,
respectively, i.e., they are the outer functions in $\mathbb D$ that
give the factorizations
\begin{equation}\label{facti}
|\Pi|^{-2}=2\pi\sigma',\quad |\~\Pi|^{-2}=2\pi\~\sigma'\,\,.
\end{equation}
In \eqref{intert}, send $m\to\infty$ to get
\begin{equation}\label{facti1}
2\Pi=\~\Pi\left(\phi_n+\phi_n^*+\~F(\phi_n^*-\phi_n)\right)\,\,.
\end{equation}
Thus, the first formula in \eqref{facti} shows that for the
sufficiently regular measures,  Steklov's condition $\sigma'\ge
\delta/(2\pi)$ is equivalent to
\begin{equation}\label{oc-1}
\left|\~\Pi\left(\phi_n+\phi_n^*+\~F(\phi_n^*-\phi_n)\right)\right|\le\frac2{\sqrt\delta},\quad
z=e^{i\theta}\in\mathbb T\,\,.\end{equation}
Since $|\phi_n|=|\phi_n^*|$ on $\mathbb T$, we have
\[
\left|\~\Pi\left(\phi_n+\phi_n^*+\~F(\phi_n^*-\phi_n)\right)\right|\le2|\~\Pi|
\left(|\phi^*_n|+|\~F(\phi^*_n-\phi_n)|\right)<2C_1(\delta)|\~\Pi|\left(\Re\~F\right)^{1/2}=2C_1(\delta)
\]
due to \eqref{main}, \eqref{fefl}, and the second formula in
\eqref{facti}. Thus, to guarantee \eqref{oc-1}, we only need to take
$C_1(\delta)=\delta^{-1/2}$ in \eqref{main}. In this section, we
assume $\delta$ to be fixed so the exact formulas for $C(\delta)$
and $C_1(\delta)$ will not be needed.

\end{proof}

\begin{proof}{\it (of Theorem \ref{T3}).} The proof will be based on the Decoupling Lemma and will contain two
parts. In the first one, we will make the choice for $\phi^*_n$,
$\widetilde F$ and study their basic properties to check conditions
$(1)$,$(3)$-$(5)$ of the Decoupling Lemma. In the second part, we
will verify the normalization condition, i.e., condition $(2)$.

\subsection{ The choice of parameters}

In what follows, we take $\e_n=n^{-1}$.\bigskip

{\bf 1.  The choice of $\~F$.} Consider two parameters:
$\alpha\in(1/2,1)$ and $\rho\in(0,\rho_0)$ where $\rho_0$ is
sufficiently small. Let us emphasize that these parameters are fixed
and will not be changed in the estimates below, however many
constants in these inequalities will actually depend on them. We do
not trace this dependence here.

 Take
\begin{equation}\label{eff}
\~F(z)=\~C_n\left(\rho(1+\e_n-z)^{-1}+(1+\e_n-z)^{-\alpha}\right),
\end{equation}
where the positive normalization constant $\~C_n$ will be chosen
later. We have two terms inside the brackets. The first one gives
the right growth at point $z=1$: $(1+\e_n-1)^{-1}=n$ and this choice
is motivated by conditions $(3)$ and $(5)$ of the Decoupling Lemma.
The role of the second term will be explained later.

 We will need more information on $\~F$. Clearly
$\~F$ is smooth and has a positive real part in $\overline{\mathbb
D}$. Notice that for $z=e^{i\theta}\!\in\!\mathbb T$ and
$|\theta|\ll 1$, we have
\begin{align*}
1\!+\!\e_n\!-\!z&=\e_n+\tfrac{\theta^2}2-i\theta+O(|\theta|^3)~,\quad|1\!+\!\e_n\!-\!z|^2=
\big(\e_n+\tfrac{\theta^2}2+O(|\theta|^3)\big)^2+\big(-\!\theta+O(|\theta|^3)\big)^2=\\&=\e_n^2+\theta^2+\e_n\theta^2
+O(|\theta|^3)\end{align*} Since
$w^{-1}=\smash{\frac{\textstyle\overline
w}{\textstyle|w|^{2\vphantom]}}}$, we obtain
\begin{equation}\label{rep-in-1}
(1\!+\!\e_n\!-\!z)^{-1}=\frac{\e_n+\frac{\theta^2}2+i\theta+O(|\theta|^3)}{\e_n^2+\theta^2+\e_n\theta^2+O(|\theta|^3)}
=\frac{\e_n+i\theta}{\e_n^2+\theta^2}+O\Big(\frac{\theta^2}{\e_n^2+\theta^2}\Big)\,\,.
\end{equation}
The last equality can be verified directly by subtraction. So,
\begin{equation}\label{Caratheodory}
\~F(e^{i\theta})=\~C_n\left(\frac{\rho\e_n}{\e_n^2+\theta^2}+i\frac{\rho\theta}{\e_n^2+\theta^2}+
(1+\e_n-z)^{-\alpha}+O\left(\frac{\rho\theta^2}{\e_n^2+\theta^2}\right)\right)\,\,.
\end{equation}
We have
\begin{equation}\label{j1}
(1+\e_n-e^{i\theta})^{-\alpha}=((1+\e_n-\cos\theta)^2+\sin^2\theta)^{-\alpha/2}\exp(-i\alpha\Gamma_n(\theta))\,\,.
\end{equation}
\begin{equation}\label{j2}
\Gamma_n(\theta)=-\arctan\left( \frac{\sin\theta}{1+\e_n-\cos\theta}\right)\in(-\pi/2,\pi/2)\,\,.
\end{equation}
The following bound is true
\[
\left\|\frac{\rho\e_n}{\e_n^2+\theta^2}+|(1+\e_n-e^{i\theta})^{-\alpha}|+
O\left(\frac{\rho\theta^2}{\e_n^2+\theta^2}\right)\right\|_{L^1[-\pi,\pi]}\sim
1
\]uniformly in $n$.
Then, for every fixed $\upsilon>0$, we have
\begin{equation}\label{ochen-prosto}
\rho(1+\e_n-z)^{-1}+(1+\e_n-z)^{-\alpha}\to
\rho(1-z)^{-1}+(1-z)^{-\alpha}, \quad n\to \infty
\end{equation}
uniformly in $\{z=e^{i\theta},\, |\theta|>\upsilon\}$. Since
\[
\int_{|\theta|>\upsilon}\Re
\left(\rho(1-e^{i\theta})^{-1}+(1-e^{i\theta})^{-\alpha}\right)d\theta\sim
1
\]
we can choose $\~C_n$ to guarantee \eqref{norka} and then $\~C_n\sim
1$ uniformly in $n$.

Consider the formulas (\ref{rep-in-1}) and (\ref{Caratheodory}).
They yield
\begin{equation}\label{tit1}
|\~F|\sim\e_n^{-1}=n \quad {\rm for }\,\,  |\theta|<\e_n
\end{equation}
and
\begin{equation}\label{tit2}
|\~F|\sim|\theta|^{-1}  \quad {\rm for }\,\, |\theta|>\e_n\,\,.
\end{equation}
Indeed, in the last inequality, the upper bound
\[
|\~F|\lesssim |\theta|^{-1}
\]
is immediate. For the lower bound,
\begin{eqnarray*}
|\~F|\ge|\Im \~F|=\left|
\frac{\rho\sin\theta}{(1+\e_n-\cos\theta)^2+\sin^2\theta}
\hspace{5cm} \right.
\\
\left.
+\left((1+\e_n-\cos\theta)^2+\sin^2\theta\right)^{-\alpha/2}\sin
\left(\alpha\arctan
\frac{\sin\theta}{1+\e_n-\cos\theta}\right)\right|\,\,.
\end{eqnarray*}
These terms have the same signs, so
\[
|\Im \~F|\ge\left|
\frac{\rho\sin\theta}{(1+\e_n-\cos\theta)^2+\sin^2\theta}
\right|=\left| \frac{\rho\sin\theta}{2(1+\e_n)(1-\cos\theta)+\e_n^2}
\right|\gtrsim |\theta|^{-1}, \quad \e_n<|\theta|<\pi/2\,\,.
\]
For $\theta: |\theta|>\pi/2$, the estimate
\[
|\~F|\sim |\theta|^{-1}
\]
is a trivial corollary of (\ref{ochen-prosto}).
\bigskip

{\bf 2. The choice of $\phi_n^*$.} Let $\phi_n^*$ be chosen as
follows
\begin{equation}\label{vybor12}
\phi_n^*(z)=C_n f_n(z), \, f_n(z)=P_m(z)+Q_m(z)+Q_m^*(z)
\end{equation}
where $P_m$ and $Q_m$ are certain polynomials of degree
\begin{equation}\label{propa}
m=[\delta_1 n]
\end{equation}
where $\delta_1$ is small and will be specified later. Notice here
that $Q_m^*$ is defined by applying the $n$--th order star
operation. The constant $C_n$ will be chosen in such a way that
\[
\int_{-\pi}^\pi |\phi^*_n|^{-2}d\theta=2\pi
\]
(i.e., \eqref{norma} is satisfied). To prove the Theorem, we only
need to show that
\begin{equation}\label{si-n}
C_n=\left(\int_{-\pi}^\pi |f_n|^{-2}d\theta\right)^{1/2}\sim1
\end{equation}
uniformly in $n$ and that $f_n$ satisfies the other conditions of
the Decoupling Lemma.

The choice of $\phi_n^*$ is motivated by the following observation.
The estimate \eqref{main} requires
\[
|\~
F(z)(\phi_n(z)-\phi_n^*(z))|<C_1(\delta)\Bigl(\Re\~F(z)\Bigr)^{1/2},
\quad z\in \mathbb{T}\,.
\]
Since $|F(z)|$ is much larger than $\Re\~F(z)$ around $z=1$, the
point of growth, the factor $\phi_n-\phi_n^*$ should provide some
cancelation. The sum of the second and the third terms in
\eqref{vybor12}, the polynomial $Q_m+Q_m^*$ has degree $n$ and is
symmetric so it drops out in $\phi_n-\phi_n^*$. However, it has
zeroes on $\mathbb{T}$ due to Lemma \ref{noli} and thus can not be a
good choice for $\phi_n^*$ due to violation of conditions $(1)$ and
$(2)$ in Decoupling Lemma. $P_m$, the first term in \eqref{vybor12},
will be be chosen to achieve a certain balance. It will be small
around $z=1$ and it will push the zeroes of $Q_m+Q_m^*$ away from
$\overline{\mathbb{D}}$ to guarantee \eqref{norma}.

\bigskip

Consider the Fejer kernel
\begin{equation}\label{feya}
\cal F_m(\theta)=\frac1m\frac{\sin^2(m\theta/2)}{\sin^2(\theta/2)}=
\frac1m\frac{1-\cos(m\theta)}{1-\cos(\theta)},\quad\cal F_m(0)=m
\end{equation}
and the Taylor approximation to the function $(1-z)^{-\alpha}$,
i.e.,
\[
R_{(k,\alpha)}(z)=1+\sum^k_{j=1}d_jz^j
\]
(see Appendixes for the detailed discussion). We define $Q_m$ as an
analytic polynomial without zeroes in $\mathbb D$ which gives
Fejer-Riesz factorization
\begin{equation}\label{fact}
|Q_m(z)|^2=\cal G_m(\theta)+|R_{(m,\alpha/2)}(e^{i\theta})|^2\,\,.
\end{equation}
\begin{equation}\label{sdvig}
\cal G_m(\theta)=\cal F_m(\theta)+\frac12\cal F_m\left(\theta-\frac\pi
m\right)+\frac12\cal F_m\left(\theta+\frac\pi m\right)\,\,.
\end{equation}
Clearly, the right hand side of \eqref{fact} is a positive
trigonometric polynomial of degree $m$ so this factorization is
possible and $Q_m$ is unique up to a unimodular factor. We choose
this factor in such a way that $Q_m(0)>0$. Since $Q_m$ is an outer
function, we have the following canonical representation (see
\cite{duren}, page 24)
\begin{equation}\label{mult-mult}
Q_m(z)=\exp\left(\frac1{2\pi}\int_{-\pi}^\pi C(z,e^{i\theta})\ln|Q_m(e^{i\theta})|d\theta\right)~,\quad|z|\!<\!1\,.
\end{equation}

Notice that $Q_m^*$ is a polynomial of degree $n$ with positive
leading coefficient. Since $|Q_m(e^{i\theta})|$ is even in $\theta$,
this representation shows that ${\cal H}(z)=\ln Q_m(z)$ is analytic
in $\mathbb D$ and has real Taylor coefficients (indeed,
$\overline{{\cal H}(z)}=\cal{H}(\overline{z})$). That, on the other
hand, implies that $Q_m(z)=e^{\cal{H}(z)}$ has real coefficients as
well.

For $P_m$, we take
\begin{equation}\label{ppp}
P_m(z)=Q_m(z)(1-z)(1-0.1R_{(m,-(1-\alpha))}(z))
\end{equation}
and $\deg P_m=2m+1<n$ by the choice of small $\delta_1$. Consequently, $\deg \phi^*_n=n$.\bigskip

Now that we have chosen $\~F$ and $\phi_n^*$, it is left to show
that they satisfy the conditions of the Decoupling Lemma. The second
term in \eqref{eff} and the structure of \eqref{ppp} will become
important in what follows.

{\bf 1. $f_n$ has no zeroes in $\overline{\mathbb D}$}.
For $f_n$, we can write
\begin{equation}\label{efn}
f_n=Q_m(z)\left((1-z)(1-0.1R_{(m,-(1-\alpha))}(z))+1+z^ne^{-2i\phi}\right),\quad
z\!\in\!\mathbb T\end{equation}
where
\[
e^{-2i\phi(\theta)}=\frac{\overline{Q_m(e^{i\theta})}}{Q_m(e^{i\theta})}
\]
so $\phi$ is an argument of $Q_m$. The polynomial $Q_m$ has no zeroes in $\mathbb D$ and
\begin{equation}\label{dunduk}
(1-z)(1-0.1R_{(m,-(1-\alpha))}(z))+1+\frac{Q_m^*}{Q_m}
\end{equation}
is analytic in $\mathbb D$ and has positive real part. Indeed,
\[
\Re\left(1+\frac{Q_m^*}{Q_m}\right)\ge 0,\quad z\!\in\!\mathbb T
\]
since $|Q_m|=|Q^*_m|$.

Since $Q_m$ has real coefficients, $Q_m(1)$ is real. Furthermore, since $Q_m(0)\!>\!0$, $Q_m$ is real at the
real line and $Q_m\!\ne\!0$ in $\mathbb D$, then $Q_m(1)\!>\!0$. So $\phi(0)=0$ and
\[
\Re\left(1+\frac{Q_m^*}{Q_m}\right)=2,\quad z=1\,\,.
\]
For the first term in \eqref{dunduk}, we have
\begin{equation}\label{parts}
\Re\big[(1-z)\big(1-0.1R_{(m,-(1-\alpha))}(z)\big)\big]=(1-\cos\theta)(1-0.1X)-0.1Y\sin\theta
\end{equation}
where
\[
X=\Re R_{(m,-(1-\alpha))},\quad Y=\Im R_{(m,-(1-\alpha))}\,\,.
\]
Notice that\begin{equation}\label{_51}
|R_{(m,-(1-\alpha))}|<2\quad\text{for }z\!\in\!\mathbb
D\,\,.\end{equation}So $|X|<2$ in $\mathbb D$ as well. The function
$Y$ is odd in $\theta$ and $Y(\theta)<0$ for $\theta\in (0,\pi)$ (as
follows from the Lemma \ref{poly1} in Appendix A for small $\theta$;
 see \cite{askey}, Theorem 7.3.5 for the general case). Thus, the function $f_n(z)$ has
a positive real part on $\mathbb T$ and, in particular, has no
zeroes in $\mathbb D$. We conclude then that $\phi_n^*$ has no
zeroes in $\mathbb D$.\smallskip

{\bf 2. The growth at $z=1$.}  The formula \eqref{efn} implies
\[
f_n(1)=2Q_m(1)\,\,.
\]
Then, \eqref{feya} and \eqref{fact} yield
\[
|f_n(1)|\gtrsim\sqrt m\sim\sqrt n
\]
due to \eqref{propa}.\bigskip

{\bf 3. Steklov's condition.}
We need to check \eqref{main} with $\phi_n$ replaced by $f_n$. From \eqref{fact} and \eqref{efn}, we get
\[
|f_n|^2\lesssim|Q_m|^2=\cal G_m(\theta)+|R_{(m,\alpha/2)}(e^{i\theta})|^2\,\,.
\]
Lemma \ref{poly2} implies
\[
|R_{(m,\alpha/2)}(e^{i\theta})|^2\lesssim(\e_m+|\theta|)^{-\alpha}\lesssim(\e_n+|\theta|)^{-\alpha}\,\,.
\]
The exact form of the Fejer's kernel \eqref{feya} gives
\[
\cal G_m(\theta)\lesssim\frac m{m^2\theta^2+1}\lesssim\frac n{n^2\theta^2+1}=\frac{\e_n}{\e_n^2+\theta^2}\,\,.
\]
We have
\begin{eqnarray*}
\Re \~F=
\frac{\rho(1+\e_n-\cos\theta)}{(1+\e_n-\cos\theta)^2+\sin^2\theta}
\hspace{5cm}
\\
+\left((1+\e_n-\cos\theta)^2+\sin^2\theta\right)^{-\alpha/2}\cos
\left(\alpha\arctan \frac{\sin\theta}{1+\e_n-\cos\theta}\right)\,\,.
\end{eqnarray*}
Since $\alpha<1$ and $|\arctan (\cdot)|<\pi/2$, we get
\[
\Re \~F\sim \frac{\rho(\e_n+\theta^2)}{\e_n^2+\theta^2}+
\frac{1}{(\e_n^2+\theta^2)^{\alpha/2}}\,\,.
\]
The  estimate
\[
\frac{\theta^2}{\e_n^2+\theta^2}\lesssim
\frac{1}{(\e_n^2+\theta^2)^{\alpha/2}}\,,
\]
uniform in $\theta$, implies
\[
\Re\~F\sim\frac{\e_n}{\e_n^2+\theta^2}+(\e_n^2+\theta^2)^{-\alpha/2}\,\,.
\]Therefore,
\begin{equation}\label{vtoro}
|f_n|\lesssim|Q_m|\lesssim(\Re\~F)^{1/2}\,\,.
\end{equation}
Then, for the second term in \eqref{main}, we get
\[
|\~F(f_n-f_n^*)|^2=|\~F(P_m-P_m^*)|^2\lesssim|\~F(1-z)|^2|Q_m|^2\,\,.
\]
The uniform bounds
\[
|\~F(1-z)|\lesssim1,~|Q_m|^2\lesssim\Re\~F
\]
together with \eqref{vtoro}, imply \eqref{main} with $\phi_n$ replaced by $f_n$.\bigskip

\subsection{Normalization: checking condition $(2)$ of the Decoupling Lemma} We only need to check now that
\[
\int_{\mathbb T}\frac1{|f_n|^2}\,d\theta\sim 1
\]
(see \eqref{si-n}). It is sufficient to consider $\theta\!\in\![0,\pi]$
 as $|f_n(e^{i\theta})|$ is even.

Our first goal is to obtain a convenient lower bound on $|f_n|^2$.
 Notice that \eqref{fact} yields
\[
|Q_m(e^{i\theta})|^2\ge |R_{(m,\alpha/2)}(e^{i\theta})|^2
\]
and the Lemma \ref{poly2} from Appendix A gives (we should use
notation $\beta\!=\!1\!-\!\alpha>0$ here)
\[
|f_n|^2\gtrsim(m^{-1}+|\theta|)^{-\alpha}\left|(1-z)(1-0.1R_{(m,-\beta)}(z))+1+z^ne^{2i\phi}\right|^2\,\,.
\]
Then, the representation \eqref{parts} leads to
\begin{equation}\label{xyz}
\left|(1-z)(1-0.1R_{(m,-\beta)}(z))+1+z^ne^{2i\phi}\right|^2=
\end{equation}
\[
\Bigl((1-\cos\theta)(1-0.1X)-0.1Y\sin\theta +1+\cos(n\theta-2\phi)\Bigr)^2+
\]
\[
\Bigl(-0.1Y(1-\cos\theta)-(1-0.1X)\sin\theta+\sin(n\theta-2\phi)\Bigr)^2\,\,.
\]
For the first term, we have
\begin{equation}\label{interm-s}
(1-\cos\theta)(1-0.1X)-0.1Y\sin\theta +1+\cos(n\theta-2\phi)\ge -0.1Y\sin\theta\gtrsim \theta^{2-\alpha}
\end{equation}
where the last inequality follows from Lemma \ref{poly1}. Thus,
\begin{equation}\label{inv}
|f_n|^2\gtrsim(m^{-1}+|\theta|)^{-\alpha}\left[
(\theta^{2-\alpha})^2 + \bigl(\Psi(\theta)+\sin(n\theta-2\phi)\bigr)^2
\right]\end{equation}
where
\[
\Psi=
\Im\big((1-z)(1-0.1R_{(m,-\beta)}(z))\big)=-0.1Y(1-\cos\theta)-(1-0.1X)\sin\theta\,\,.
\]
In what follows, we will control $\Psi$ and $\phi$ to analyze
$\Psi(\theta)+\sin(n\theta-2\phi)$ in \eqref{inv}. We will locate
the zeroes $\{\theta_j\}$ of this highly oscillatory function and
will show that away from these points the normalization condition is
easily satisfied. More delicate analysis will be needed to integrate
$|f_n|^{-2}$ over small neighborhoods of $\{\theta_j\}$.

To bound $|\Psi'(\theta)|$, we use  Lemma \ref{der-der} and
\eqref{_51}. Indeed,
\begin{gather*}
|\Psi'(\theta)|=\left|\Im\frac{d~}{d\theta}\Bigl[(1-z)\,(1-0.1R_{m,-\beta}(z))\Bigr]\right|\le\\
\le\left|(1-0.1R_{m,-\beta}(e^{i\theta}))\right|
+0.1\left|(1-e^{i\theta})\frac{d~}{d\theta}
R_{m,-\beta}(e^{i\theta})\right|\le\\\le(1+0.1\cdot2)+0.1\cdot2|\theta|\cdot
O\left(\max\Big(\frac1n,|\theta|\Big)^{\beta-1}\right)=O(1)\;.\end{gather*}
Then,
\begin{equation}\label{len-1}
|\Psi'(\theta)|\lesssim1
\end{equation}
 uniformly in $n$. This estimate and $\Psi(0)=0$
imply
\begin{equation}\label{aPs}|\Psi(\theta)|\lesssim|\theta|\end{equation}
by integration.

For the phase $\phi$, we have $\phi(0)=0$  and
\[
|\phi'(\theta)|\lesssim m
\]
where the last inequality is proved in Appendix B. Since we have the
derivative of $\phi$ under control,
\begin{equation}\label{nOm}
(n\theta-2\phi(\theta))'=n+O(\delta_1n)\,\,.
\end{equation}
By making $\delta_1$ small, we can make sure that the function $n\theta-2\phi(\theta)$ is
 monotonically increasing and
\begin{equation}\label{n2}
n/2<(n\theta-2\phi(\theta))'<2n\,\,.
\end{equation}

To study the zeroes $\{\theta_j\}$, we first introduce auxiliary
points $\{\widehat \theta_j\}$. The monotonicity of
$n\theta-2\phi(\theta)$ allows us to uniquely define $\{\widehat
\theta_j\}$ as solutions to the equation:
\begin{equation}\label{tthat}
~n\widehat\theta_j-2\phi(\widehat\theta_j)=\pi(j-\tfrac12), \,
j=0,\ldots, [cn]\,.
\end{equation}
Then, $\sin\big(n\widehat\theta_j-2\phi(\widehat\theta_j)\big)=(-1)^{j-1}$. On the other hand, from \eqref{n2}, one has
\begin{equation}\label{porovnu}
\frac{(2j\!-\!1)\pi}{4n}\!<\!\widehat\theta_j\!<\!\frac{2\pi
j}n\;,\end{equation} By  the estimate \eqref{aPs}, we can choose
some small positive constant $c\!>\!0$ so that for
$j=0,1,\cdots,[cn]$ the expression
$\Psi(\theta)+\sin(n\theta-2\phi(\theta))$ changes the sign on
$[\widehat\theta_j,\widehat\theta_{j+1}]$. Indeed,
$$
(-1)^j\!\!\left[\Psi(\widehat\theta_j)+\sin\big(n\widehat\theta_j\!-\!2\phi(\widehat\theta_j)\big)\right]\!=
(-1)^j\big[\Psi(\widehat\theta_j)\!+\!(-1)^{j-1}\big]\!<
\big|\Psi(\widehat\theta_j)\big|\!-\!1=O\Big(\frac{2\pi j}n\Big)\!-\!1<0\;;
$$$$
(-1)^j\!\!\left[\Psi(\widehat\theta_{j+1})+\sin\big(n\widehat\theta_{j+1}\!-\!2\phi(\widehat\theta_{j+1})
\big)\right]\!>\!1\!-\!\big|\Psi(\widehat\theta_{j+1})\big|=1\!-\!O\Big(\frac{2\pi j}n\Big)>0\;.
$$
Moreover, we require that $c$ is chosen such that
\[
|\Psi(\widehat\theta_j)|<\frac1{\sqrt2}, \quad j=0,1,\ldots, [cn]
\] and that $\upsilon$,
defined as
\[
\upsilon=\widehat\theta_{[cn]},
\]
is smaller than the parameter $\upsilon$ from the Lemmas
\ref{poly2}, \ref{derider}, \ref{poly1} in the Appendix A. Now, let
us show that there is the unique point $\theta_j$ such that
\begin{equation}\label{tt}
\theta_j\!\in\![\widehat\theta_j,\widehat\theta_{j+1}]:~\Psi(\theta_j)+
\sin\big(n\theta_j-2\phi(\theta_j)\big)=0\,\,.
\end{equation}
The existence of such $\theta_j$ is a simple corollary of continuity
and  sign change.

Note, that the function $~\Xi(\theta)\!=n\theta-2\phi(\theta)-\pi j~$, restricted to
the segment $[\widehat\theta_j,\widehat\theta_{j+1}]$, takes values
from $[-\frac\pi2,\frac\pi2]\vphantom{\sum\limits^-_-}$. So, for
each $\theta_j$ that satisfies \eqref{tt}, we have
$~\sin\big(\Xi(\theta_j)\big)\!=\!(-1)^{j-1}\Psi(\theta_j)~$ and
\begin{equation}\label{ttj}
n\theta_j-2\phi(\theta_j)+(-1)^j\arcsin\big(\Psi(\theta_j)\big)=\pi j\;.
\end{equation}
If, for fixed $j$, there are several solutions $\theta_j$ to
\eqref{tt}, then the derivative of the function
\[
n\theta-2\phi(\theta)+(-1)^j\arcsin\big(\Psi(\theta)\big)
\]
 in the left hand
side of \eqref{ttj} is non-positive at at least one of these
$\theta_j$. However the lower estimate on the derivative  reads
$$
n-2\phi'(\theta_j)+\frac{(-1)^j\Psi'(\theta_j)}{\sqrt{1-\Psi^2(\theta_j)}}>
n-O(m)-\frac{O(1)}{\sqrt{1-\frac12}}>\frac n2\;
$$
which shows that $\theta_j$ is unique.

Since $|\Psi(\widehat\theta_j)|<\frac1{\sqrt2}$, we have
$\arcsin|\Psi(\widehat\theta_j)|<\frac\pi4$. Now, from \eqref{nOm},
\eqref{tthat} and \eqref{ttj}, one gets
$$
\theta_j\!-\!\widehat\theta_j>\frac1{2n}\Big(\frac\pi2\!-\!\frac\pi4\Big)~,~
\widehat\theta_{j+1}\!-\!\theta_j>\frac1{2n}\Big(\frac\pi2\!-\!\frac\pi4\Big)\text{~,~i.e.,}\quad
I_j\!=\Big[\theta_j\!-\!\frac{0.01}n,\theta_j\!+\!\frac{0.01}n\Big]
\subset\big[\widehat\theta_j,\widehat\theta_{j+1}\big]\;.
$$
If $\theta\!\in\![\widehat\theta_j,\widehat\theta_{j+1}]\backslash
I_j$, then
\begin{equation}\label{vne}
|(n\theta\!-\!2\phi(\theta))-(n\theta_j\!-\!2\phi(\theta_j))|\!>\!\frac
n2\cdot\frac{0.01}n>\frac1{200}\,\,.
\end{equation}
The estimate for the integral over $I_0\!=[0,0.01n^{-1}]$ is easy
since on that arc we have  $|Q_m|\sim m$ (by \eqref{fact})  and
$$|f_n|^2\gtrsim m\big(1\!+\!\cos(n\theta\!-\!2\phi(\theta))\big)^2$$
by \eqref{efn}, \eqref{interm-s}. The bound on the derivative of
$\phi$ implies that $|f_n|\sim\sqrt m$ on $I_0$ so
\[
\int_{I_0}|f_n|^{-2}d\theta\sim n^{-2}\,\,.
\]
Let $\theta\,{\in}\,I_j$ and assume that some $\xi$ is located
between $\theta_j$ and $\theta$. From the definition of $I_j$, we
get $|\xi\,{-}\,\theta_j|\,{<}\,\frac1{100n}$. Therefore,
\begin{gather*}
|n\xi\!-\!2\phi(\xi)\!-\!\pi j|\le|n\theta_j\!-\!2\phi(\theta_j)\!-\!\pi j|+
|(n\xi\!-\!2\phi(\xi))-(n\theta_j\!-\!2\phi(\theta_j))|\le\\\le
|\arcsin\Psi(\theta_j)|+2n|\xi\!-\!\theta_j|\le\frac\pi4+\frac2{100}\;;
\end{gather*}
and
\begin{equation}
|\cos(n\xi-2\phi(\xi))|\ge\cos\Big(\frac\pi4+\frac2{100}\Big)>\frac12\;.
\end{equation}
So, for $\theta\!\in\!I_j,~j\!\ge\!1$, we have
\[
|\Psi(\theta)+\sin(n\theta-2\phi(\theta))|=\Bigl|\int_{\theta_j}^\theta\Bigl(\Psi'(\xi)+\cos(n\xi-2\phi(\xi))
(n-2\phi'(\xi)\Bigr)d\xi\Bigr|\sim n|\theta-\theta_j|.
\]
For $\theta\!\in\![\widehat\theta_j,\widehat\theta_{j+1}]\backslash
I_j$, $j\!\ge\!1$, one gets
\begin{gather*}
|\sin(n\theta\!-\!2\phi(\theta))\!+\!\Psi(\theta)|\ge
|\sin(n\theta\!-\!2\phi(\theta))-
\sin(n\theta_j\!-\!2\phi(\theta_j))|
-|\Psi(\theta)\!-\!\Psi(\theta_j)|
\end{gather*}
by the triangle inequality.
  Then, for the last term, we apply (\ref{len-1}) to get
\[
|\Psi(\theta)\!-\!\Psi(\theta_j)|\lesssim |\theta-\theta_j|\le
|\widehat\theta_{j+1}-\widehat\theta_j|=O(n^{-1})\,\,.
\]
For the first one,
\[
\sin x_1-\sin
x_2=2\sin\left(\frac{x_1-x_2}{2}\right)\cos\left(\frac{x_1+x_2}{2}\right)
\]
where $x_1=n\theta-2\phi(\theta)$ and
$x_2=n\theta_j-2\phi(\theta_j)$. Next, we use (\ref{vne}) to write
\[
\frac{\pi}{2}\ge|x_1-x_2|\ge\frac{1}{200}\,\,.
\]
Consider $(x_1+x_2)/2$. For $x_2$, we have $x_2=\pi j-(-1)^j\arcsin
\Psi(\theta_j)$ with $|\arcsin\Psi(\theta_j)|\le\pi/4$. Then, for
every $x_1\in [\pi j-\pi/2,\pi j+\pi/2]$, we get
\[
\left|\cos\left(\frac{x_1+x_2}{2}\right)\right|>\cos(3\pi/8)\,\,.
\]
Therefore,
$$|\Psi(\theta)+\sin\big(n\theta\!-\!2\phi(\theta)\big)|\sim1$$
outside $\cup_j I_j$.
\bigskip

Now, let us obtain the estimates outside the small fixed arc
$\{\theta: |\theta|>\upsilon\}$. The bound (\ref{efn}) implies
\[
|Q_m|^2\Bigl[\Re
\Bigl((1-z)(1-0.1R_{(m,-(1-\alpha))}(z))\Bigr)\Bigr]^2\le
|f_n|^2\le|Q_m|^2(2+|(1-z)(1-0.1R_{(m,-(1-\alpha))}(z))|)^2\,\,.
\]
We have the uniform convergence
\[
R_{(m,-(1-\alpha))}(z)\to (1-z)^{1-\alpha}, \quad m\to \infty,\,
z\in \mathbb{T}
\]
and
\[
|Q_m(z)|\to |1-z|^{-\alpha}, \quad m\to \infty, \quad
z=e^{i\theta},\, |\theta|>\upsilon.
\]
The direct calculation shows (see, e.g., (\ref{interm-s})) that
\[
\Re \Bigl((1-z)(1-0.1R_{(m,-(1-\alpha))}(z))\Bigr)\sim 1, \quad
|\theta|>\upsilon\,\,.
\]
Consequently,
\begin{equation}\label{ravno-1}
\int_{|\theta|>\upsilon} |f_n|^{-2}d\theta\sim C(\upsilon,\alpha)\,\,.
\end{equation}
 From \eqref{inv}, we have
\begin{gather*}
\|f_n^{-1}\|_2^2\lesssim\int\limits_{[0,\upsilon]\cap(\cup_j I_j)^c}
\theta^\alpha d\theta+\sum_{j=1}^{[cn]}~\theta_j^\alpha\hskip-3ex\int\limits^{0.01n^{-1}}_{-0.01n^{-1}}
\frac{d\theta}{(\theta_j^{2-\alpha})^2+n^2\theta^2}+n^{-2}+
\int_{|\theta|>\upsilon} |f_n|^{-2}d\theta\lesssim\\
\lesssim1+\frac1n\sum_{j=1}^{[cn]}\theta_j^{-2+2\alpha}\sim
1+\int_0^1\theta^{2\alpha-2}d\theta\lesssim1
\end{gather*}
where we used \eqref{porovnu}, \eqref{tt}, and
$\alpha\!\in\!(1/2,1)$. Together with (\ref{ravno-1}), that implies
\[
\int_{\mathbb{T}} |f_n|^{-2}d\theta\sim 1
\]
and the proof of Theorem \ref{T3} is finished.
\end{proof}

{\bf Remark.} It is immediate from the proof that the constructed
polynomial
 satisfies the following bound:
\begin{equation}\label{pol-sv}
|\phi_n(e^{i\theta},\sigma)|^2\le
C(\alpha)\left(\frac{n}{1+(n\theta)^2}+\frac{1}{|\theta|^\alpha}\right),
\quad \alpha\in (0.5,1)\,.
\end{equation}

\vspace{0.5cm}

\section{Measure of orthogonality}

 Our method allows one to compute a measure of orthogonality
$\sigma$ for which the orthonormal polynomial has the required size
and it is interesting to compare it to the results on the maximizers
we obtained before. The calculations given below will show that
$\sigma$ is purely absolutely continuous. Its density can be
represented as a sum of background ${\cal B}(\theta)$, $0<C_1\le
{\cal B}(\theta)\le C_2$, and a combination of ``peaks" positioned
at $\widetilde \theta_j$ to be defined later. Qualitatively, each
peak resembles the mollification of the point mass by the Poisson
kernel. Our analysis can establish the parameters of mollification
and a ``mass" assigned to each peak.

The formulas \eqref{facti} and \eqref{facti1} yield
\[
\sigma'=\frac{4\~\sigma'}{|\phi_n+\phi_n^*+\~F(\phi_n^*-\phi_n)|^2}=\frac{2\Re\~F}
{\pi|\phi_n+\phi_n^*+\~F(\phi_n^*-\phi_n)|^2}\,\,.
\]
This expression is explicit as we know the formulas for all functions involved. We have
\begin{equation}\label{chto}
\sigma'^{-1}=\frac{C_n|Q_m|^2}{\Re\~F}\cdot\left|
2(1+e^{i(n\theta-2\phi)})+H_n(e^{i\theta})+e^{i(n\theta-2\phi)}\overline{H_n(e^{i\theta})}+
\~F\left(H_n(e^{i\theta})-e^{i(n\theta-2\phi)}\overline{H_n(e^{i\theta})}\right)\right|^2
\end{equation}
where
\begin{equation}\label{eto-ash}
H_n(z)=(1-z)(1-0.1R_{(m,-(1-\alpha))}(z)),\quad C_n\sim1
\end{equation}
as follows from \eqref{si-n}, \eqref{ppp}.

Consider the first factor. We can apply \eqref{tit1}, \eqref{tit2},
and Lemma \ref{poly2} to get
\[
\frac{|Q_m|^2}{\Re\~F}\sim1\,\,.
\]
Recall that $\~F$ is given by
\begin{equation}\label{small-w}
\~F(z)=\~C_n(\rho(1+\e_n-z)^{-1}+(1+\e_n-z)^{-\alpha}),\,\~C_n=(\rho/(1+\e_n)+(1+\e_n)^{-\alpha})^{-1}
\end{equation}
where the last formula for $\~C_n$ comes from the normalization
\eqref{norka} and
\[
\frac{1}{2\pi}\int_{\mathbb{T}} \Re \widetilde
F(e^{i\theta})d\theta=\Re \widetilde F(0)
\]
(the Mean Value Formula for a harmonic function continuous in
$\overline{\mathbb{D}}$).

Substitution into the second factor in \eqref{chto} gives
\begin{equation}\label{nasha-r}
(\sigma')^{-1}\sim\left|(2+\overline
H_n(1-\~F))\left(e^{i(n\theta-2\phi)}+\frac{2+H_n(1+\~F)}{2+\overline
H_n(1-\~F)}\right)\right|^2\,\,.
\end{equation}
For $z=e^{i\theta}$ and small positive $\theta$ (the negative values
can be handled similarly), we have
\[
H_n(e^{i\theta})=-i\theta+0.1\theta i\;\!R_{(m,-(1-\alpha))}(e^{i\theta})+O(|\theta|^2)
\]
and Lemma \ref{poly1} can be used  for $R_{(m,-(1-\alpha))}$. Next,
consider $\~F$. It can be written as
\[
\~F(e^{i\theta})=\~C_n\left(\frac{\rho}{\e_n-i\theta}+\frac1{(\e_n-i\theta)^\alpha}\right)+O(1)
\]
Therefore, for the first factor in \eqref{nasha-r}, we have
\begin{equation}\label{mesto-odin}
|2+\overline H_n(1-\~F)|\sim1
\end{equation}
when $\rho\!\in\!(0,\rho_0(\alpha))$ and $\rho_0(\alpha)$ is small.

Consider
\[
J=\frac{2+H_n(1+\~F)}{2+\overline H_n(1-\~F)}\,\,.
\]
Notice first that $|J|\!>\!1$ for $\theta\!\neq\!0$. Indeed,
\begin{eqnarray*}
|J|^2-1=\frac{|2+H_n(1+\~F)|^2-|2+\overline
H_n(1-\~F)|^2}{|2+\overline
H_n(1-\~F)|^2}=\frac{4\Re\left((H_n+\overline
H_n+|H_n|^2)\overline{\~F}\right)}{|2+\overline
H_n(1-\~F)|^2}=\\=\frac{4(\Re\~F)(|H_n|^2+2\Re H_n)}{|2+\overline
H_n(1-\~F)|^2}
\end{eqnarray*}
and the last expression is positive by the choice of $\~F$ and $H_n$. For $\theta=0$, we have $J(0)=1$.

For small $\theta$, the following asymptotics holds
\begin{equation}\label{masa1}
|J|^2-1\sim
|\theta|^{2-\alpha}\left(\frac{\rho\e_n}{\e_n^2+\theta^2}+|\theta|^{-\alpha}\right),\quad|\theta|>0.1n^{-1}
\end{equation}
and
\[
|J|^2-1\sim\left(\frac{\rho\e_n}{\e_n^2+\theta^2}+\e_n^{-\alpha}\right)|\theta|n^{\alpha-1},\quad|\theta|<0.1n^{-1}\,\,.
\]
If we write
\[
J=r(\theta)e^{i\Upsilon(\theta)}, \quad \Upsilon(0)=0
\]
then
\begin{equation}\label{merka}
(\sigma')^{-1}\sim|e^{i(n\theta-2\phi(\theta)-\Upsilon(\theta))}+r(\theta)|^2=
(\cos(n\theta-2\phi(\theta)-\Upsilon(\theta))+r(\theta))^2+\sin^2(n\theta-2\phi(\theta)-\Upsilon(\theta))\,\,.
\end{equation}

Consider $\{\widetilde\theta_j\}$, the solutions to
\[
n\theta-2\phi(\theta)-\Upsilon(\theta)=\pi+2\pi j,\quad|j|<N_3
\]
that belong to some small fixed arc $|\theta|<\upsilon$. We have $\Upsilon(0)=0$ and the direct estimation gives
\[
|\Upsilon'(\theta)|<0.1n
\]
uniformly for all $\theta$. Indeed, it is sufficient to show that
\begin{equation}\label{1976-1}
\left|\partial_\theta\left(\frac{2+H_n(1+\~F)}{2+\overline
H_n(1+\overline{\~F})}\right)\right|<0.01n
\end{equation}
and
\begin{equation}\label{1976-2}
\left|\partial_\theta\left(\frac{2+\overline
H_n(1-\~F)}{2+H_n(1-\overline{\~F})}\right)\right|<0.01n\,\,.
\end{equation}
The both inequalities are proved in Lemma \ref{faza-5} from
Appendix B.

Now we can argue that the distance between the consecutive
$\{\widetilde\theta_j\}$ is of size $n^{-1}$ and
\[
\sigma'\sim\frac{1}{(r(\widetilde
\theta_j)-1)^2+n^2(\theta-\widetilde\theta_j)^2}= \widetilde
m_j\frac{\widetilde y_j}{\widetilde
y_j^{\,2}+(\theta-\widetilde\theta_j)^2}
\]
on $\theta: |\theta-\widetilde\theta_j|<Cn^{-1}$. In the Poisson
kernel, the mollification parameter $\widetilde y_j$ is
\[
\widetilde y_j=\frac{r(\widetilde \theta_j)-1}{n}
\]
and the mass $\widetilde m_j$ is given by
\[
\widetilde m_j=\frac{1}{n(r(\widetilde\theta_j)-1)}
\]
Notice, that
\[
\sum_{j=-N_3}^{N_3} \widetilde m_j\lesssim
\int_{0.1n^{-1}<|\theta|<\upsilon}
\frac{d\theta}{r(\theta)-1}<\infty
\]
as follows from \eqref{masa1} and $\alpha\in (1/2,1)$. Away from
these $\{\widetilde\theta_j\}$ the density is $\sim 1$.

{\bf Remark.} In the estimates above, we assumed that $\rho$ is
small: $\rho\in (0,\rho_0)$. The choice of $\rho_0$ is made in Lemma
\ref{faza-5} (see the Remark after it) and in \eqref{mesto-odin}.
Thus, we first fix a parameter $\alpha$ and then fix  $\rho$. In
fact, we need $\rho$ to be small only to control $\sigma'$ and it is
irrelevant for the proof of the main Theorem. \vspace{0.5cm}

The measure $\sigma$ constructed in the proof has no singular part.
Its regularity can be summarized in the following Lemma. For
$\delta\in (0,1), p\in [1,\infty], C>0$, let us introduce the
following class of measures given by a weight
\[
S_{\delta}^{(p,C)}=\{\sigma: d\sigma=w(\theta)d\theta, w\ge
\delta/(2\pi), \|w\|_1=1, \|w\|_p\le C\}
\]
and let
\[
M_{n,\delta}^{(p,C)}=\sup_{\sigma\in S_{\delta}^{(p,C)}}
\|\phi_n(z,\sigma)\|_{L^\infty(\mathbb{T})}
\]

\begin{lemma}\label{build-up}
For every $p\in [1,\infty)$ there is $\delta_0\in (0,1)$ and $C>0$
such that
\[
M_{n,\delta_0}^{(p,C)}\sim \sqrt n
\]
\end{lemma}
\begin{proof}
The upper estimate is obvious since $M_{n,\delta}^{(p,C)}\le
M_{n,\delta}$. For the lower one, we only need to show that for
every large $p$ one can take $\alpha$ in the proof of Theorem
\ref{T3} so close to 1 that $\|w\|_p<C$ where $w=\sigma'$ and $C$
 is independent of $n$. The estimate \eqref{merka} and the bounds
 on the derivatives of $\phi$ and $\Upsilon$ yield
 \[
 \int_{\mathbb{T}} w^pd\theta \lesssim 1+\sum_{j=1}^{N_3}
\int_0^{Cn^{-1}} \frac{d\theta}{((r(\widehat
\theta_j)-1)^2+n^2\theta^2)^p}\lesssim 1+\sum_{j=1}^{N_3} \frac{1}{n
(r(\widehat\theta_j)-1)^{2p-1}}
 \]
 \[
 \lesssim
 1+\int_0^1\frac{d\theta}{|\theta|^{(2p-1)(2-2\alpha)}}<\infty
 \]
 if $\alpha\in ((4p-3)/(4p-2),1)$. Here we used \eqref{masa1} to estimate $r(\theta)-1$.
\end{proof}

\bigskip
{\Large\part{Bernstein's method and localization.  The proofs of
Theorem~\ref{T3-i} and Theorem~\ref{rrra-i}}\bigskip}

In this part, we will use the ``localization principle" to first
prove the lower bounds on $M_{n,\delta}$ in the full range of
$\delta$ (Theorem \ref{T3-i}) and then iterate this construction and
prove Theorem \ref{rrra-i}.\bigskip\bigskip

\section{The method by Bernstein and  localization principle.}
Given a weight $w$ on $[-\pi,\pi]$, we define
\[
\lambda(w)=\exp\left(\frac{1}{4\pi}\int_{\mathbb{T}} \ln(2\pi
w(\theta)) d\theta\right), \quad \Lambda(w)=\sqrt{\|w\|_1}\,.
\]
We have
\[
\phi_n(z,w)=\|w\|_1^{-1/2}\phi_n(z,w/\|w\|_1)
\]
and \eqref{gabor1}, \eqref{gabor2} yield
\[
\exp\left(\frac{1}{4\pi}\int_{\mathbb{T}} \ln(2\pi
w(\theta)/\|w\|_1) d\theta\right)\le
\left|\frac{\Phi_n(z,w)}{\phi_n(z,w/\|w\|_1)}\right|\le 1\,.
\]
So,
\begin{equation}\label{sravni}
\lambda(w)\le\left|\frac{\Phi_n(z,w)}{\phi_n(z,w)}\right|\le
\Lambda(w)\,.
\end{equation}
In \cite{bernstein}, S. Bernstein studied the asymptotics of the
polynomials when the weight of orthogonality is regular and
introduced a method which we will use when proving the following
Theorem.
\begin{theorem}\label{lemma21}
Let $w_{1(2)}$ be two weights on $[-\pi,\pi]$ so that
\begin{equation}\label{c2}
w_1(\theta)=w_2(\theta), \quad \theta\in [-\epsilon,\epsilon]\,.
\end{equation}
Then
\begin{equation}\label{localization}
\left|\frac{\phi_n(1,w_1)}{\phi_n(1,w_2)}\right|\le
\frac{\Lambda(w_2)}{\lambda(w_1)}+\frac{4\Lambda(w_1)}{\epsilon\lambda(w_1)}
\left(\int_{|\theta|>\epsilon}|\phi_n(e^{i\theta},w_1)\phi_n(e^{i\theta},w_2)|(w_1+w_2)d\theta\right)
\end{equation}
for all $n$.
\end{theorem}
\begin{proof}

 Following Bernstein, we write
\begin{equation}\label{lenin}
\Phi_n(z,w_1)=\varrho_n
\phi_n(z,w_2)+\sum_{j=0}^{n-1}\phi_j(z,w_2)\int_{-\pi}^\pi\Phi_n(e^{i\theta},w_1)
\overline{\phi}_j(e^{i\theta},w_2)w_2(\theta)d\theta
\end{equation}
with some coefficient $\varrho_n$.  By orthogonality,
\[
\Phi_n(z,w_1)=\varrho_n
\phi_n(z,w_2)+\sum_{j=0}^{n-1}\phi_j(z,w_2)\int_{-\pi}^\pi\Phi_n(e^{i\theta},w_1)
\overline{\phi}_j(e^{i\theta},w_2)\Bigl(w_2(\theta)-w_1(\theta)\Bigr)d\theta=
\]
\[
\varrho_n \phi_n(z,w_2)+\int_{-\pi}^\pi\Phi_n(e^{i\theta},w_1)
K_{n-1}(e^{i\theta},z,w_2)
\Bigl(w_2(\theta)-w_1(\theta)\Bigr)d\theta\,.
\]
The Christoffel-Darboux kernel $K_{n-1}$ admits a representation
(see, e.g., \cite{murad}, p. 225, formula~(8.2.1)):
\[
K_{n-1}(\xi,z,\mu)=\frac{\phi_n^*(z,\mu)\overline{\phi_n^*}(\xi,\mu)-\phi_n(z,\mu)\overline{\phi_n}(\xi,\mu)}{1-z\overline
\xi}\,.
\]
Then,
\[
|\Phi_n(1,w_1)|\le |\phi_n(1,w_2)|\left(
\varrho_n+4\epsilon^{-1}\int_{|\theta|>\epsilon}|\Phi_n(e^{i\theta},w_1)|\cdot
|\phi_n(e^{i\theta},w_2)|(w_1+w_2)d\theta \right)\,.
\]
We will now use \eqref{sravni}. Comparing the coefficients in front
of $z^n$ in \eqref{lenin}, we get
\[
\lambda(w_2)\le \varrho_n\le \Lambda(w_2)
\]
and so
\[
\left|\frac{\phi_n(1,w_1)}{\phi_n(1,w_2)}\right|\le
\frac{\Lambda(w_2)}{\lambda(w_1)}+\frac{4\Lambda(w_1)}{\epsilon\lambda(w_1)}
\left(\int_{|\theta|>\epsilon}|\phi_n(e^{i\theta},w_1)\phi_n(e^{i\theta},w_2)|(w_1+w_2)d\theta\right)
\]
by the repetitive application of \eqref{sravni}.
\end{proof}

\bigskip
\section{The proofs of Theorem \ref{T3-i} and Theorem \ref{rrra-i}.}

We start with a Lemma which will immediately imply Theorem
\ref{T3-i}. It allows to perturb very general measures and have the
orthogonal polynomial grow.

\begin{lemma}\label{vozmuse}
Assume $\delta\in (0,1], p\in [2,\infty)$ and the weight $w$
satisfies the following properties:
\[
\|w\|_1=1, \quad w\ge \delta/(2\pi), \quad w\in L^p[-\pi,\pi]\,.
\]
Then, for arbitrary $\epsilon>0$ and $n\in \mathbb{N}$, there is a
weight $\widetilde w$ such that
\begin{equation}\label{trebovanie}
\|\widetilde w\|_1=1, \quad \widetilde w\ge \delta/(2\pi)-\epsilon,
\quad \|w-\widetilde w\|_p\le \epsilon
\end{equation}
and
\[
|\phi_n(1,\widetilde w)|>C(\epsilon,\delta,p,w)\sqrt n\,.
\]
\end{lemma}

\begin{proof} Take any $\delta\in
(0,1]$. For every $p_1\in (p,\infty)$, Lemma \ref{build-up} yields
$\sigma_1: d\sigma_1=\sigma'_1d\theta$ so that
\[
\sigma'_1\ge \delta_0/(2\pi)>0,\quad
\|\sigma'_1\|_{L^{p_1}(\mathbb{T})}<C(p_1), \quad
\|\sigma_1'\|_{L^1(\mathbb{T})}=1, \quad
|\phi_n(1,\sigma_1)|>C(p_1)\sqrt n\,.
\]
The constants $C(p_1)$ above are $n$--independent. Consider an
interval $I_\tau=(-\tau,\tau)$.  Then,
\[
\int_{I_\tau}d\sigma_1\lesssim \tau^{1/p_1'}\,.
\]
We now introduce two new weights $w_2,\widetilde w$ given by:
\[
 \quad w_2=\left\{
\begin{array}{cc}
\sigma'_1/\delta_0, & \theta\in I_\tau\\
w, & \theta\notin I_\tau
\end{array}
\right., \quad \widetilde w=w_2/{\|w_2\|_1}\,.
\]
We have $w_2\ge (2\pi)^{-1}\delta$ a.e. on $\mathbb{T}$ and
\[
\|w-w_2\|_p\le \|w\|_{L^p[-\tau,\tau]}+\|w_2\|_{L^p[-\tau,\tau]}\le
o(1)+ C(p_1)\tau^{p^{-1}-p_1^{-1}}
\]
by H\"older's inequality. Here $o(1)\to 0$ as $\tau\to 0$.
Therefore,
\[
\|w-w_2\|_1\lesssim o(1)+C(p_1)\tau^{p^{-1}-p_1^{-1}}\,.
\]
The triangle inequality and normalization $\|w\|_1=1$ give
\[
\|w_2\|_1=1+o(1)+O(\tau^{p^{-1}-p_1^{-1}})\,.
\]
We can choose $\tau$ small enough that the last two conditions in
\eqref{trebovanie} are satisfied for $\widetilde w$.  For the
corresponding polynomials, we have
\[
\phi_n(1,\sigma_1/\delta_0)=\sqrt{\delta_0}\phi_n(1,\sigma_1)
\]
so $|\phi_n(1,\sigma_1/\delta_0)|\gtrsim  \sqrt { n}$. Apply Lemma
\ref{lemma21} with $w_1=\sigma_1'/\delta_0$ and $w_2$ defined above.
Notice that $0<C_1(\delta,p)\le \lambda(w_{1(2)})\le C_2(\delta,p),
C_1(\delta,p)\le \Lambda(w_{1(2)})\le C_2(\delta,p)$ and $w_1=w_2$
on $I_\tau$ by construction. We have
\[
\int_{|\theta|>\tau}|\phi_n(e^{i\theta},w_1)\phi_n(e^{i\theta},w_2)|w_2d\theta\lesssim
\max_{I_\tau^c} |\phi_n(e^{i\theta},w_1)|
\cdot\|w_2\|_{L^2(\mathbb{T})}\cdot
\|\phi_n(e^{i\theta},w_2)\|_{L^2(\mathbb{T})}
\]
and
\[
\int_{|\theta|>\tau}|\phi_n(e^{i\theta},w_1)\phi_n(e^{i\theta},w_2)|w_1d\theta\le
\max_{I_\tau^c} |\phi_n(e^{i\theta},w_1)|
\cdot\|w_1\|_{L^2(\mathbb{T})}\cdot
\|\phi_n(e^{i\theta},w_2)\|_{L^2(\mathbb{T})}\,.
\]
The polynomials orthonormal with respect to a measure in Steklov
class have uniformly bounded $L^2(\mathbb{T})$ norm (see the
 proof of Lemma \ref{ots-sv}). For the estimation of $\max_{I_\epsilon^c}
|\phi_n(e^{i\theta},w_1)|$, we use \eqref{pol-sv} to get
\[
\max_{I_\tau^c} |\phi_n(e^{i\theta},w_1)|\lesssim
\left(\frac{n}{1+n^2\tau^2}+\tau^{-\alpha}\right)^{1/2}\lesssim
\tau^{-1/2}
\]
where $\alpha<1$. Thus, the localization principle
\eqref{localization} gives
\[
|\phi_n(1,\widetilde w)|\ge C(\epsilon,\delta,p,w)\sqrt n
\]
and the proof is finished.
\end{proof}

Now the proof of the Theorem \ref{T3-i} is immediate.
\begin{proof}{\it (of Theorem \ref{T3-i})}
It is sufficient to take $w=(2\pi)^{-1}$ and $\delta=1$.
\end{proof}

{\bf Remark.} Notice that this  proof allows to improve Lemma
\ref{build-up} to cover the full range of $\delta: \delta\in (0,1)$.
This statement is much stronger than the Theorem \ref{T3-i} itself:
it shows that $\sqrt n$ growth can be achieved on far more regular
weights.

\flushleft Now, we can iterate this construction to prove Theorem
\ref{rrra-i}.

\begin{proof}{\it(of Theorem \ref{rrra-i}).}
Fix any $\delta\in (0,1)$ and a sequence $\{\beta_n\}:
\lim_{n\to\infty}\beta_n=0$. We can assume without loss of
generality that $\beta_1<1$. Choose any $p\in [2,\infty)$ and
parameter $\widetilde C>1$. We construct the sequence of weights
$\{w_n\}$ through the following induction:
\begin{itemize}
\item{\bf First step:} We let $w_1=(2\pi)^{-1}$ and $k_1=1$. Then,
$|\phi_{k_1}(1,w_1)|=1>\beta_1$.
\item{\bf Inductive assumption:} We assume that the weight $w_{n}$
and the natural numbers $k_1< \ldots< k_{n}$ are given so that
\[
|\phi_{k_j}(1,w_n)|>\beta_{k_j}\sqrt{k_j}, \quad j=1,\ldots, n
\]
and
\begin{equation}\label{predpol}
\|w_n\|_1=1,\quad  \|w_n\|_p<\widetilde C,\quad w_n\ge
(\delta+(1-\delta)2^{-n})/(2\pi)\,.
\end{equation}
\item{\bf Inductive step:}
For every $\epsilon>0$ and $N$ we can use the perturbation Lemma
\ref{vozmuse} to get $w_{n+1}$ so that
\[
\|w_{n+1}\|_1=1, \quad w_{n+1}\ge
(\delta+(1-\delta)2^{-n})/(2\pi)-\epsilon, \quad \|w_{n+1}-w_n
\|_p\le \epsilon
\]
and
\[
|\phi_{N}(1,w_{n+1})|>C(\epsilon,\delta,p)\sqrt N\,.
\]
Notice that for fixed $j$ the functional $\phi_j(1,\sigma)$ is
continuous in $\sigma$ in weak--$(\ast)$ (and then in
$L^p(\mathbb{T})$) topology. The second inequality in
\eqref{predpol} is strict. So, we first choose $\epsilon$ so small
that:

\begin{itemize}
\item [1.] $\|w_n\|_p+\epsilon<\widetilde C$.

\item[2.] $|\phi_{k_j}(1,\nu)|>\beta_{k_j}\sqrt{k_j}, \,
\forall j=1\ldots n, $ as long as $\nu: \|\nu-w_n\|_p<\epsilon$.

\item[3.]
$(\delta+(1-\delta)2^{-n})/(2\pi)-\epsilon>(\delta+(1-\delta)2^{-n-1})/(2\pi)$.
\end{itemize}
Then, with fixed $\epsilon$, take $N$ large so that
 $N>k_{n}$ and $C(\epsilon,\delta,p)\sqrt N>\beta_N\sqrt N$. We can always achieve that since $\lim_{l\to\infty} \beta_l=0$.
 Now, let
 $k_{n+1}=N$.

\end{itemize}
Thus, we constructed the new weight $w_{n+1}$ and $k_{n+1}$ that
satisfy all induction assumptions. At each step when going from $n$
to $n+1$ we choose new $\epsilon$ that depends on $n$, the step of
induction.

By construction, $\|w_{n+1}-w_n\|_p\lesssim 2^{-n}$ so $w_n$
converges to some $w$ in $L^p(\mathbb{T})$ norm. Moreover, $w\ge
\delta/(2\pi)$ a.e. on $\mathbb{T}$. We use the continuity of
$\phi_j(1,\sigma)$ in $\sigma$ again to get
\[
|\phi_{k_j}(1,w)|\ge \beta_{k_j}\sqrt{k_j}, \quad \forall j
\]
and that finishes the proof.
\end{proof}

{\bf Remark.} It is clear that our construction allows to have the
polynomials grow simultaneously at any finite number of points on
the circle. We also can make the measure of orthogonality symmetric
with respect to both axis $OX$ and $OY$. Indeed, the measure we
constructed in the Theorem is given by the even weight $w$. Now, for
every $N\in \mathbb{N}$, we can take $w_N(x)=w(Nx)$ and then
\[
\phi_{Nj}(z,w_N)=\phi_j(z^N,w)
\]
To make the measure symmetric with respect to both $OX$ and $OY$, it
is sufficient to take $N=2$.

{\bf Remark.} The conjecture of Steklov and its solution can be
interpreted as follows. It is known that  $\{\Phi_n(z)\}$ satisfy
the recursion
\[
\Phi^*_{n+1}(z)=\Phi_n^*(z)-\gamma_n z\Phi_n(z), \quad \Phi^*_0(z)=1\,\,.
\]
Therefore,
\[
\Phi_n^*(z)=1-z\sum_{j=0}^{n-1}\gamma_j \Phi_j(z)\,\,.
\]
Recall that $\sigma\in S_\delta$ implies $\{\gamma_j\}\in \ell^2$
and one can define a maximal function in analogy to the Carleson
maximal function in Fourier series, i.e.,
\[
\cal{M}(\theta)=\sup_{n} \left|\sum_{j=0}^n \gamma_j
\Phi_j(e^{i\theta})\right|, \quad \cal{M}(\theta)\sim
\sup_{n}|\Phi_n^*(e^{i\theta})-1|\,\,.
\]
Then, for the example we constructed,
\[
\sum_{j=0}^\infty \gamma_j \Phi_j(e^{i\theta})\in
L^\infty(\mathbb T) \quad {\rm but} \quad {\cal
M}(e^{i\theta})\notin L^\infty(\mathbb T)\,\,.
\]
\bigskip

{\Large \part{Applications}}\smallskip

In this part, we apply the obtained results to handle the case of
the orthogonality on the segment on the  real line. We also prove
the sharp bounds for the polynomial entropy in the Steklov
class.\bigskip

\section{Back to the real line.}

In the case when the measure $\sigma$ is symmetric on $\mathbb T$
with respect to the real line, one can relate $\{\phi_n(z,\sigma)\}$
to polynomials orthogonal on the real line through the following
standard procedure. Let $\psi, (x\in [-1,1], \psi(-1)=0 )$ be a
non-decreasing bounded function with an infinite number of growth
points. Consider the system of polynomials $\{P_k\}, (k=0,1,\ldots
)$ orthonormal with respect to the measure $\psi$ supported on the
segment $[-1,1]$. Introduce the function

\begin{equation}
\sigma(\theta)=\left\{
\begin{array}{cc}
-\psi(\cos \theta), & 0\le \theta\le \pi, \\
\psi(\cos \theta), & \pi \le \theta \le 2\pi,%
\end{array}%
\right.\label{trans}
\end{equation}
which is bounded and non-decreasing on $[0,2\pi]$. Consider the polynomials $%
\phi_k(z,\sigma)=\lambda_k z^k+\cdots,$
$(\lambda_k=(\rho_0\cdot\ldots\cdot \rho_{k-1})^{-1}>0)$ orthonormal
with respect to $\sigma$. Then, $\phi_n$ is related to $P_k$ by the
formula
\begin{equation}\label{real-line}
P_k(x,\psi)=\frac{\phi_{2k}(z,\sigma)+\phi^{*}_{2k}(z,\sigma)}{\sqrt{2\pi
\left[ 1+\lambda_{2k}^{-1}\phi_{2k}(0,\sigma)\right]}}\,
z^{-k},\quad k=0,1,\ldots,
\end{equation}
where $x\!=\!(z\!+\!z^{-1})/2$ (\cite{5,6}). This reduction also
works in the opposite direction: given the symmetric measure
$\sigma$ we can map it to the measure on the real line and the
corresponding polynomials will be related by \eqref{real-line}. We
are ready to formulate the Theorem.
\begin{theorem}\label{rrra-i-r}
Let $\delta\!>\!0$.  Then, for every positive sequence $\{\beta_n\}:
\lim_{n\to\infty}\beta_n=0$, there is a measure $\psi:d\psi=\psi'dx$
supported on $[-1,1]$ such that $\psi'(x)\geq \delta$ for a.e. $x\in
[-1,1]$ and

\begin{equation}\label{est-ra-i-r}
|P_{k_n}(0,\psi)|\gtrsim \beta_{k_n} \sqrt{{k_n}}
\end{equation}
for some sequence $\{k_n\}\subset \mathbb{N}$.
\end{theorem}
\begin{proof}
Indeed, in the Theorem \ref{rrra-i} we can take $\sigma^*:
d\sigma^*=w^*d\theta$ to be symmetric with respect to both axis,
i.e., $w^*(2\pi-\theta)=w^*(\theta)$ (symmetry with respect to $OX$)
and $w^*(\theta)=w^*(\pi-\theta)$ (symmetry with respect to $OY$).
Moreover, we can always arrange for all $\{k_n\}$ to be divisible
by~$4$ and
\[
|\phi_{k_n}(1,\sigma^*)|\geq \beta_{k_n}\sqrt{k_n}
\]
Now, we take $\sigma: d\sigma=w^*(\theta-\pi/2)d\theta$, i.e., the
rotation of $\sigma^*$ by $\pi/2$ and apply \eqref{trans} to it. The
symmetries of $\sigma^*$ yield the symmetry of $\sigma$ with respect
to $OX$ so this transform is applicable. Notice that
$\phi_n(z,\sigma)=e^{in\pi/2}\phi_n(ze^{-i\pi/2},\sigma^*)$ where
the first factor is introduced to make the leading coefficient
positive. Also, notice that $\phi_n^*(1,\sigma^*)$ is real-valued so
$ \phi_n(1,\sigma^*)=\phi_n^*(1,\sigma^*) $. We have
\[
\phi_{k_n}(i,\sigma)+\phi^*_{k_n}(i,\sigma)=\phi_{k_n}(i,\sigma)+\overline{\phi_{k_n}(i,\sigma)}=
\]
\[
\phi_{k_n}(1,\sigma^*)+\overline{\phi_{k_n}(1,\sigma^*)}=2\phi_{k_n}(1,\sigma^*)\geq
2\beta_{k_n}\sqrt{k_n}
\]

 Now, notice that
\[
\phi_{2k}(0,\sigma)/\lambda_{2k}=\Phi_{2k}(0,\sigma)=-\overline{\gamma}_{2k-1}\,\,.
\]
By the Szeg\H{o} sum rule, $\gamma_n\in \ell^2$ and so
$\lim_{n\to\infty}\gamma_n= 0$. Thus, \eqref{real-line} gives
\[
|P_{k_n/2}(0,\psi)|\gtrsim \beta_{k_n} \sqrt{{k_n}}
\]
and, after redefining $k_n$, this is exactly \eqref{est-ra-i-r}. For
the derivative of $\psi$, we have
\[
\psi'(\cos\theta)=\frac{\sigma'(\theta)}{|\sin\theta|}>\frac{\delta_1/(2\pi)}{|\sin
\theta|}>\delta_1/(2\pi)=\delta
\]
if $\delta_1=2\pi\delta$.
\end{proof}
{\bf Remark.} The original conjecture of Steklov was formulated in
terms of the weights (i.e., the unit ball in $L^1(\mathbb T$)) and
we solved it in that form. However, as the results on  maximizers
from  the first part of the paper suggest, the class of probability
measures is far more natural for that setting.
\bigskip\bigskip

\section{ The polynomial entropies and the Steklov class.}
In recent years, a lot of efforts were made (see, e.g., \cite{ap1,
ap2, ap3}) to study the so-called polynomial entropy
\[
\int_\mathbb T |\phi_n|^2\ln|\phi_n|d\sigma
\]
where $\phi_n$ are orthonormal with respect to $\sigma$.
Since $\smash{\sup\limits_{x\in [0,1]}} x^2|\ln x|<\infty$, this quantity is bounded if and only if
\[
\int_\mathbb T
|\phi_n|^2\ln^+|\phi_n|d\sigma
\]
is bounded. The last expression is important as it contains the information on the size of $\phi_n$.
In this section, we consider the following variational problem
\[
\Omega_n(\cal K)=\sup_{\sigma\in \cal{K}} \int_\mathbb T
|\phi_n|^2\ln^+|\phi_n|d\sigma
\]
where $\phi_n$ is the $n$-th orthonormal polynomial with respect to
$\sigma$ taken in $\cal K$, some special class of measures. It is an
interesting question to describe those $\cal{K}$ for which
$\Omega_n(\cal{K})$ is bounded in $n$. So far, this is known only
for very few $\cal{K}$, e.g., the Baxter class of measures. For the
Szeg\H{o} class with measures normalized by the $\ell^2$ norm of
Schur parameters, the sharp estimate $\Omega_n\sim \sqrt n$ is known
\cite{dk}. In this section, we will obtain the sharp bound on
$\Omega_n(S_\delta)$.
\begin{lemma}
If $\delta\!\in\!(0,1)$, then
\[
\Omega_n(S_\delta)\sim \ln n\,\,.
\]
\end{lemma}
\begin{proof}
If one takes the measure $\sigma$ and the polynomial $\phi_n$
constructed in the proof of the Theorem~\ref{T3-i}, then
\[
\Omega_n(S_\delta)\gtrsim
\int_{\mathbb T}|\phi_n|^2\ln^+|\phi_n|d\theta\gtrsim 1+
\int_{0.01n^{-1}}^{\theta_1-0.01n^{-1}} |f_n|^2\ln^+|f_n|d\theta
\]
where $\theta_1$ was introduced in this proof. On that interval,
$|\Psi(\theta)+\sin(n\theta-2\phi)|>C$ and so $|f_n|\sim |Q_m|$.
This follows from \eqref{efn} and the verification of the
normalization condition. Then, the expression \eqref{uh-uh} gives a
very rough lower bound
\[
|Q_m(e^{i\theta})|^2\gtrsim  \frac{m}{m^2\theta^2+1}+1\,\,.
\]
This shows $|Q_m|\sim \sqrt n$ on the interval $(0.01n^{-1},\theta_1-0.01n^{-1})$ and so
\[
\Omega_n(S_\delta)\gtrsim n^{-1}\cdot n\cdot \ln n\sim \ln n\,\,.
\]
Therefore, the polynomial entropy grows at least as the logarithm.
On the other hand, the trivial upper bound $
\|\phi_n\|_\infty\lesssim \sqrt n $ implies that $
\Omega_n(S_\delta)\lesssim \ln n $.
\end{proof}
\vspace{0.5cm}

{\Large\part*{Some open problems}}\bigskip

In conclusion, we want to discuss some interesting problems we
didn't address.\bigskip

\begin{enumerate}
\item{In the variational problem for $M_{n,\delta}$, it would be interesting to know whether
the maximizer is unique and how many mass points it possesses.
Ideally, one would  want to find it explicitly. At the moment, very
little is know about the maximizers $\mu^\ast$ in
Theorem~\ref{spec-forma}. In \cite{den-pams}, it was proved that
$N$, the number of point masses, is of order $n$.

}\smallskip

\item{Suppose that $\phi_n(z,\sigma)$ is the orthonormal polynomial, $\sigma\in S_\delta$, and
$|\phi_n(1,\sigma)|>C\sqrt n$. What is the behavior of the Schur
parameters $\{\gamma_j(\sigma)\}$? This question is interesting as
its answer can give a ``difference equation perspective" to the
problem. To this end, one only needs to find the coefficients of the
Szeg\H{o} recursion (Schur parameters) such that
\[
\limsup_{j\to\infty}
\|\phi_j(e^{i\theta},\sigma)\|_{L^\infty(\mathbb T)}<C
\]
(which is equivalent to the Steklov condition if $\sigma$ is
regular) but
\[\|\phi_n(e^{i\theta},\sigma)\|_{L^\infty(\mathbb T)}\sim \sqrt n\,\,.\]
for the fixed arbitrarily large $n$. }\smallskip

\item{For the following variational problem
\[
M_{n,\delta_1,\delta_2}=\sup_{\delta_1/(2\pi)\le
\sigma'\le\delta_2/(2\pi)}
\|\phi_n(e^{i\theta},\sigma)\|_{L^\infty(\mathbb{T})}
\]
find the sharp estimates for $M_{n,\delta_1,\delta_2}$ as
$n\to\infty$.  }
\end{enumerate}

\vspace{0.5cm}

{\Large \part*{Acknowledgement.}} The research of S.D. was supported
by NSF grant DMS-1067413. The research of A.A. and D.T. was
supported by the grants RFBR 13-01-12430 OFIm and RFBR 11-01-00245
and Program~1 DMS RAS. The hospitality of IMB (Bordeaux) and IHES
(Paris) is gratefully acknowledged by S.D. The authors thank Stas
Kupin, Fedor Nazarov, and Evguenii Rakhmanov for interesting
comments.\vspace{1cm}

\part*{Appendix A.}\bigskip

In this Appendix, we start by introducing the polynomials that
approximate the function $(1-z)^{-\alpha/2}$ (used in the formula
\eqref{fact}) and the function $(1-z)^{1-\alpha}$ (used in the
definition of $P_m$, formula \eqref{ppp}). These polynomials are
well studied (see, e.g., \cite{zygmud}, chapter 5) but we deduce the
necessary estimates here for completeness of exposition.  Notice
first, that $(1-z)^\beta$ is analytic in $\mathbb D$ and has
positive real part for any $\beta\!\in\!(-1,1)$. For
$z=e^{i\theta}\!\in\!\mathbb T$, we have
\[
(1-z)^\beta=\Bigl((1-\cos\theta)^2+\sin^2\theta\Bigr)^{\beta/2}\exp(-i\beta
L(\theta))=\big(2\sin\frac{|\theta|}2\big)^\beta\exp(-i\beta L(\theta))
\]
where
\[L(\theta)=
\arctan\Big(\frac{\sin\theta}{1-\cos\theta}\Big)=\arctan\big(\!\cot\tfrac{\textstyle\theta}{\textstyle2\vphantom]}
\big)\]
and so
\begin{equation}\label{chisto}
(1-z)^\beta=|\theta|^{\beta}(1+O(\theta^2))\exp(-i\beta L(\theta))\,,\quad\theta\to0~;\quad
L(\theta)=\operatorname{sign}(\theta)\frac{\pi-|\theta|}2~.
\end{equation}
We will now introduce the polynomials that  approximate
$(1-z)^\beta$ uniformly on compacts in $\mathbb D$ and behave on the
boundary in a controlled way. We will treat the cases of positive
and negative $\beta$ separately. Let $A_n(z)$ be the $n$-th Taylor
polynomial of $(1-z)^{\beta}$ with positive $\beta$, i.e.,
\[
A_n(z)=1-\sum^n_{j=1}c_j z^j~,\quad
c_j=c_j(\beta)=\frac{\beta(1\!-\!\beta)\ldots(j\!-\!1\!-\!\beta)}{j!}\,\,.
\]
The polynomial $R_{(n,-(1-\alpha))}$ in the main text will be taken as $A_n$ with $\beta=1-\alpha\in(0,1/2)$.

For $B_n(z)$, we choose $n$--th Taylor coefficient of
$(1-z)^{-\beta}$ with positive $\beta$, i.e.,
\[
B_n(z)=1+\sum^n_{j=1}d_j z^j
\]
and
\begin{equation}\label{d_j}
d_j=d_j(\beta)=
\frac{\beta(\beta\!+\!1)\ldots(\beta\!+\!j\!-\!1)}{j!}=\frac{j^{\beta-1}}{\Gamma(\beta)}+O(j^{\beta-2})\,\,.
\end{equation}
The polynomial $R_{(n,\alpha/2)}$ used in the main text is $B_n$ with $\beta=\alpha/2\in(1/4,1/2)$.

We need the following simple Lemmas.
\begin{lemma}\label{trifle}
For any $a>0$, we have
\[
\int^a_0\frac{\cos x}{x^\gamma}dx>0,{\rm~if}\quad\gamma\!\in\![1/2,1)
\]
and
\[
\int^a_0\frac{\sin x}{x^\gamma}dx>0,~\int^\infty_0\frac{\sin
x}{x^\gamma}dx>0,{\rm~if}\quad\gamma\!\in\!(0,1)\,\,.
\]
\end{lemma}\begin{proof}
The inequalities with $\sin$ are elementary as $x^{-\gamma}$ decays and $\sin x$ satisfies
\[\sin(\pi+x)=-\sin x;\quad\sin x>0,\,x\!\in\!(0,\pi)\]
For the first inequality, we notice that
\[
\int_{3\pi/2}^a\frac{\cos x}{x^\gamma}dx>0
\]
for any $a>3\pi/2$ and we only need to show that
\[
\int_0^{3\pi/2}\frac{\cos x}{x^\gamma}dx>0\,\,.
\]
Integrating by parts we have
\[
\int_0^{3\pi/2} \frac{\cos
x}{x^\gamma}dx=-\left(\frac2{3\pi}\right)^\gamma+\gamma\int_0^{3\pi/2}\frac{\sin
x}{x^{\gamma+1}}dx>-\left(\frac2{3\pi}\right)^\gamma+\frac{2\gamma}\pi
\int_0^{\pi/2}x^{-\gamma}dx
\]
where  we dropped the integral over $[\pi/2,3\pi/2]$ in the last
inequality and used the fact that $x^{-1}\sin x$ decays
monotonically on $[0,\pi/2]$. Calculating the integral, we get
\[
\frac{2\gamma}{\pi(1-\gamma)}\left(\frac{\pi}{2}\right)^{1-\gamma}-\left(\frac2{3\pi}\right)^\gamma=
\Big(\frac2\pi\Big)^\gamma\Big(\frac\gamma{1-\gamma}-3^{-\gamma}\Big)>0
\]
for $\gamma\!\in\![1/2,1)$.\end{proof}

Let us first study the properties of $B_n$. As $B_n$ is the Taylor expansion of $(1-z)^{-\beta}$ and
$\beta\!\in\!(0,1/2)$, we have the uniform convergence $B_n(z)\!\to\!(1-z)^{-\beta}$ in
$\{|z|\le 1\}\cap\{|1-z|>1-\upsilon\}$ for any fixed $\upsilon\!>\!0$ as long as $n\!\to\!\infty$.
Indeed, due to monotonicity of $d_j$ we have
$${\textstyle\sum\limits^n_{j=1}}|d_{j+1}-d_j|=|d_{n+1}-d_1|\,\,.$$
Then, the Abel's transform yields the uniform convergence.

We now take $z=e^{i\theta}$ with $\theta\in(-\upsilon,\upsilon)$ where $\upsilon$ is small.\smallskip

We will need to use the following approximations by the integrals. Let $\gamma\in (0,1)$.
\begin{equation}\label{repa1}
\int\limits^n_1\frac{\cos(x\theta)}{x^\gamma}dx=\sum_{j=1}^{n-1}\int\limits_j^{j+1}
\frac{\cos(x\theta)}{x^\gamma}dx=\sum_{j=1}^{n-1}\frac1{j^\gamma}\int\limits_j^{j+1}{\cos(x\theta)}dx+
\sum_{j=1}^{n-1}\int\limits_j^{j+1}\cos(x\theta)\left(\frac1{x^\gamma}-\frac1{j^\gamma}\right)dx\,\,.
\end{equation}
Since
\begin{equation}\label{motr}
\max_{x\in[j,j+1]}|x^{-\gamma}-j^{-\gamma}|\lesssim j^{-\gamma-1}
\end{equation}
the second term is $O(1)$ uniformly in $\theta$ and $n$ and that gives
\begin{gather*}
\int^n_1\frac{\cos(x\theta)}{x^\gamma}dx=O(1)+\sum_{j=1}^{n-1}\frac1{j^\gamma}
\frac{\sin(\theta/2)}{\theta/2}\cos(j\theta+\theta/2)=\\=O(1)+\sum_{j=1}^{n-1}\frac1{j^\gamma}
\frac{\sin(\theta/2)}{\theta/2}\Bigl(\cos(j\theta)\cos(\theta/2)-\sin(j\theta)\sin(\theta/2)\Bigr)\,\,.
\end{gather*}
Similarly
\begin{equation}\label{repa2}
\int\limits^n_1\frac{\sin(x\theta)}{x^\gamma}dx=\sum_{j=1}^{n-1}\int\limits_j^{j+1}
\frac{\sin(x\theta)}{x^\gamma}dx=\sum_{j=1}^{n-1}\frac1{j^\gamma}\int\limits_j^{j+1}{\sin(x\theta)}dx+
\sum_{j=1}^{n-1}\int\limits_j^{j+1}\sin(x\theta)\left(\frac1{x^\gamma}-\frac1{j^\gamma}\right)dx\,\,.
\end{equation}
By \eqref{motr}, the second term is $o(1)$ as $\theta\!\to\!0$,
uniformly in $n$. Therefore, we have
\[
\int^n_1\frac{\sin(x\theta)}{x^\gamma}dx=o(1)+\sum_{j=1}^{n-1}\frac1{j^\gamma}\int_j^{j+1}
\sin(x\theta)\,dx=
\]
\[
o(1)+\sum_{j=1}^{n-1} \frac1{j^\gamma}
\frac{\sin(\theta/2)}{\theta/2}\sin(j\theta+\theta/2)=
\]
\[
o(1)+\sum_{j=1}^{n-1} \frac1{j^\gamma}
\frac{\sin(\theta/2)}{\theta/2}\Bigl(\sin(j\theta)\cos(\theta/2)+\cos(j\theta)\sin(\theta/2)\Bigr)\,\,.
\]
Above, $O(1)$ and $o(1)$ are written for $\theta\to 0$ and they are
uniform in $n$. Now, representations \eqref{repa1} and \eqref{repa2}
yield the formulas for
\[
\sum_{j=1}^{n-1}\frac{\cos(j\theta)}{j^\gamma},\quad\sum_{j=1}^{n-1}\frac{\sin(j\theta)}{j^\gamma}
\]
i.e.,
\begin{equation}\label{sin-a}
\sum_{j=1}^{n-1}\frac{\cos(j\theta)}{j^\gamma}=O(1)+C_{11}(\theta)\int^n_1\frac{\cos(x\theta)}{x^\gamma}dx+
C_{12}(\theta)\int^n_1\frac{\sin(x\theta)}{x^\gamma}dx
\end{equation}
and
\begin{equation}\label{cos-a}
\sum_{j=1}^{n-1}\frac{\sin(j\theta)}{j^\gamma}=o(1)+C_{21}(\theta)\int^n_1\frac{\sin(x\theta)}{x^\gamma}dx+
C_{22}(\theta)\int^n_1\frac{\cos(x\theta)}{x^\gamma}dx
\end{equation}
where $C_{11}\!\to\!1,~C_{12}\!\to\!0,~C_{21}\!\to\!1,~C_{22}\!\to\!0$ as $\theta\!\to\!0$ uniformly in $n$.\smallskip

Now we are ready for the next Lemma.
\begin{lemma}\label{poly2}
Let $\beta\!\in\!(0,1/2)$ and $\upsilon$ is sufficiently small fixed positive number, then
\[
\Re B_n(e^{i\theta})\sim(n^{-1}+|\theta|)^{-\beta},\quad\theta\!\in\!(-\upsilon,\upsilon)
\]
and\begin{alignat*}2
&\frac{\Im B_n(e^{i\theta})}{\operatorname{sign}(\theta)}\sim|\theta|^{-\beta},&\quad&0.01n^{-1}<|\theta|<\upsilon,\\
\quad&\frac{\Im B_n(e^{i\theta})}{\theta}\sim n^{1+\beta},&&|\theta|<0.01n^{-1}\,\,.
\end{alignat*}\end{lemma}
\begin{proof}
The case $|\theta|<0.01n^{-1}$ follows from \eqref{d_j} since
$\cos(j\theta)\sim 1$ and $\sin(j\theta)/(j\theta)\sim 1$. For the
other $\theta$, we first notice that it is sufficient to consider
$\theta\in (0.01n^{-1},\upsilon)$ and that \eqref{d_j} gives
\[
B_n(e^{i\theta})=1+\frac1{\Gamma(\beta)}\left(\sum^n_{j=1}j^{-1+\beta}e^{i\theta j}+O(1)\right)\,\,.
\]
Let $\gamma=1-\beta\in (1/2,1)$ and use the formulas \eqref{sin-a} and \eqref{cos-a}. Notice that
\[
\int^n_1\frac{\cos(x\theta)}{x^\gamma}dx=\theta^{\gamma-1}\int_\theta^{n\theta}\frac{\cos
t}{t^\gamma}dt\sim \theta^{\gamma-1}\quad({\rm any}~\gamma\!\in\!(1/2,1))
\]
as long as $\theta\!\in\!(0.01n^{-1},\upsilon)$. That follows from
the Lemma \ref{trifle}. The last estimate is valid for sufficiently
small $\upsilon$. Indeed, taking $\cal F$ as
$${\cal F}(t)=\int_0^t\frac{\cos u}{u^\gamma}du~,$$
we get the following bounds
$$
\int_\theta^{n\theta}\frac{\cos t}{t^\gamma}dt={\cal
F}(n\theta)-{\cal F}(\theta)\ge\min_{[0.01,\infty]}{\cal
F}-\max_{[0,\upsilon]}{\cal F}\ge \operatorname{const}>0\,.
$$
Similarly
\[
\int^n_1\frac{\sin(x\theta)}{x^\gamma}dx=\theta^{\gamma-1}\int_\theta^{n\theta}\frac{\sin
t}{t^\gamma}dt\sim\theta^{\gamma-1}\quad({\rm any}~\gamma\!\in\!(0,1))\,\,.
\]

That finishes the proof.\end{proof}
\begin{lemma}\label{derider}
For any $\beta\!\in\!(0,1)$, we have
\[
|B_n'(e^{i\theta})|\lesssim\left\{\begin{array}{cc}
|\theta|^{-1}n^\beta,&|\theta|>n^{-1}\\
n^{1+\beta},&|\theta|<n^{-1}
\end{array}\right.
\]
\[
|B''_n(e^{i\theta})|\lesssim\left\{\begin{array}{cc}
|\theta|^{-1}n^{\beta+1},&|\theta|>n^{-1}\\
n^{2+\beta},&|\theta|<n^{-1}
\end{array}\right.
\]
where the derivative is taken in $\theta\!\in\!(-\upsilon,\upsilon)$.
\end{lemma}

\begin{proof}
For $|\theta|<n^{-1}$, this follows from
\[
B_n'=\sum^n_{j=1}ijd_je^{ij\theta},~B_n''=\sum^n_{j=1}(ij)^2d_je^{ij\theta}
\]
by estimating the absolute values of each term.

For $|\theta|\!>n^{-1}$, we can use Abel's Lemma. Indeed,
\[
|B'_n|=\left|\sum^n_{j=1}jd_je^{ij\theta}\right|\lesssim
\left|\sum^n_{j=1}e^{ij\theta}j^{\beta}\right|+\sum^n_{j=1}j^{\beta-1}\,\,.
\]
The second term in the sum is bounded by $Cn^{\beta}$. For the first one, we have
\[
\left|\sum^n_{j=1}e^{ij\theta}j^{\beta}\right|\lesssim
n^\beta|S_n|+\left|\sum^n_{j=1}S_j j^{\beta-1}\right|,\quad
S_j=\sum^n_{j=1}e^{ij\theta},\quad|S_j|\lesssim|\theta|^{-1}
\]
and that yields the bound for $B_n'$. The second derivative can be estimated similarly.
\end{proof}

Next, we will study the polynomial $A_n$. For the Taylor expansion of $(1\!-\!z)^\beta$, we have
\[
(1-z)^\beta=1+\sum^\infty_{j=1}\frac{(-1)^j\beta(\beta\!-\!1)\ldots(\beta\!-\!(j\!-\!1))}{j!}z^j=
1-\sum^\infty_{j=1}c_jz^j\,\,.
\]
The coefficients  behave as follows
\begin{equation}\label{c_j}
c_j=\frac{\beta(1-\beta)\ldots(j-1-\beta)}{j!}=\frac{-j^{-\beta-1}}{\Gamma(-\beta)}+O(j^{-\beta-2})>0\,\,.
\end{equation}
The series $\sum_j c_j$ converges absolutely and
\[
1-\sum_{j=1}^n c_jz^j\to (1-z)^\beta, \quad |z|<1
\]
therefore \[1-\sum_{j=1}^\infty c_jz^j=(1-z)^\beta\] on $\mathbb{T}$
pointwise. In particular,
$$\sum^\infty_{j=1}c_j=1\,\,.$$
Therefore, the formula for $A_n$ can be rewritten as
\[
A_n(z)=M_n+\sum^n_{j=1}c_j(1-z^j)~,\quad
M_n=\!\!\sum^\infty_{j=n+1}\!\!c_j=\frac{n^{-\beta}}{\Gamma(1\!-\!\beta)}+O(n^{-\beta-1})\,\,.
\]
We again notice that $A_n(z)$ converges to $(1\,{-}\,z)^\beta$
uniformly in $\{|z|\!\le\!1\}$. Indeed $|A_n(z)|\,{<}\,2$ uniformly
in $\mathbb D$ and
$\big|(1\!-\!z)^\beta-A_n(z)\big|\!=\!\Big|\!\sum\limits^\infty_{j=n+1}\!\!c_jz^j\Big|\!<\!M_n\,$.

\begin{lemma}\label{poly1}
Let $\beta\!\in\!(0,1)$. We have
\begin{equation}\label{odin}
\Re A_n(e^{i\theta})\sim(n^{-1}+|\theta|)^{\beta},\quad\theta\!\in\!(-\upsilon,\upsilon)
\end{equation}
and
\begin{equation}\label{dva}
-\frac{\Im A_n(e^{i\theta})}{\operatorname{sign}(\theta)}\sim\left\{\begin{array}{cc}
|\theta|n^{1-\beta},&|\theta|\!<\!0.01n^{-1}\\|\theta|^{\beta}\,,&0.01n^{-1}\!<\!|\theta|\!<\!\upsilon
\end{array}\right.\,\,.\end{equation}
\end{lemma}
\begin{proof}
We only need to handle positive $\theta$. Again, if $0\!<\!\theta\!<\!0.01n^{-1}$, the estimate is simple.
\[
\Re A_n=M_n+\sum^n_{j=1}c_j(1-\cos(j\theta))
\]
and we have a bound
\[n^{-\beta}\sim M_n\le\Re A_n\lesssim M_n+\sum^n_{j=1}j^{-\beta-1}(j^2\theta^2)\lesssim
n^{-\beta}+n^{-\beta}\,\,.
\]
Similarly
\[
\Im A_n=-\sum^n_{j=1}c_j\sin(j\theta)
\]
and
\[
\sum^n_{j=1}c_j\sin(j\theta)\sim\theta\sum^n_{j=1}jc_j\lesssim\theta n^{1-\beta}\,\,.
\]
For $\theta$ from $[0.01n^{-1},\upsilon]$, we can again approximate by the integrals. We have
\[
\sum^n_{j=1}c_j\sin(j\theta)=\frac{-1}{\Gamma(-\beta)}\sum^n_{j=1}j^{-\beta-1}\sin(j\theta)+
O\bigg(\sum^n_{j=1}j^{-\beta-2}(j\theta)\bigg)\,\,.
\]
The last term is $O(\theta)$. Then, take
\[
\int\limits^n_1\frac{\sin(x\theta)}{x^{1+\beta}}dx=\sum^{n-1}_{j=1}\int\limits_j^{j+1}
\frac{\sin(x\theta)}{x^{1+\beta}}dx=\sum^{n-1}_{j=1}j^{-\beta-1}\!\!\int\limits_j^{j+1}\!\!\sin(x\theta)dx+
O\bigg(\sum\limits^{n-1}_{j=1}\!\Big(\frac1{j^{\beta+1}}-\frac1{(j\!+\!1)^{\beta+1}}\Big)(j\!+\!1)\theta\bigg)\,\,.
\]The second term is
$$
O\bigg(\sum\limits^{n-1}_{j=1}\!j^{-\beta-1}\theta\bigg)=O(\theta)\;.
$$
For the first sum, we have
\[
\sum^{n-1}_{j=1}j^{-\beta-1}\int_j^{j+1}{\sin(x\theta)}dx=\sum^{n-1}_{j=1}
j^{-\beta-1}\frac{\sin(\theta/2)}{\theta/2}\sin(j\theta+\theta/2)=
\]
\[
\frac{\sin(\theta/2)\cos(\theta/2)}{\theta/2}\sum^{n-1}_{j=1}j^{-\beta-1}\sin(j\theta)+
\frac{\sin^2(\theta/2)}{\theta/2}\sum^{n-1}_{j=1}j^{-\beta-1}\cos(j\theta)\,\,.
\]
The second term is $O(\theta)$ and
\[
\frac{\sin(\theta/2)\cos(\theta/2)}{\theta/2}\sim 1
\]
for $\theta\!\in\!(0,\upsilon)$. Then,
\[
\int^n_1\frac{\sin(x\theta)}{x^{1+\beta}}dx=\theta^\beta
\int_\theta^{n\theta}\frac{\sin x}{x^{1+\beta}}dx=\int_0^{n\theta}\frac{\sin x}{x^{1+\beta}}dx+O(\theta)\,\,.
\]
Notice that
\[
C>\int^a_0\frac{\sin x}{x^{1+\beta}}dx>\delta_2>0
\]
for any $a>0.01$ and so we have
\[
\sum^n_{j=1}c_j\sin(j\theta)\sim\theta^{\beta}+O(\theta)\sim\theta^\beta\,\,.
\]
This implies \eqref{dva}. For the real part,
\[
\Re A_n(e^{i\theta})=M_n+T_n(\theta)+O\bigg(\theta\sum^n_{j=1}j^{-1-\beta}\bigg),~T_n=
\sum^n_{j=1}\frac{1-\cos(j\theta)}{j^{1+\beta}}\,\,.
\]
The last term is $O(\theta)$. For $T_n$, we have
\[
T_n(0)=0,\quad T_n'(\theta)=\sum^n_{j=1}\frac{\sin(j\theta)}{j^\beta}\,\,.
\]
If $\theta\!\in\!(0,0.01n^{-1})$, then $T_n'\sim\theta n^{2-\beta}$.
For $\theta\!\in\!(0.01n^{-1},\upsilon)$, the formula \eqref{cos-a}
gives
\[
T_n'\sim\theta^{\beta-1}\,\,.
\]
Integration yields
\[
T_n(\theta)=\int^\theta_0T_n'(\xi)d\xi\sim\theta^\beta,\quad\theta\!\in\!(0.01n^{-1},\upsilon)\,\,.
\]
That finishes the proof.\end{proof} It is instructive to compare the
results of Lemmas \ref{poly2} and \ref{poly1} with
\eqref{chisto}.\smallskip

For the derivative of $A_n$ in $\theta$, we have
\[
A_n'=-i\sum^n_{j=1}jc_je^{ij\theta}\,\,.
\]

\begin{lemma}\label{der-der}
If $~\beta\!\in\!(0,1)$, then
\[
|A_n'|\lesssim\left\{\begin{array}{cc}
|\theta|^{\beta-1},&|\theta|>0.01n^{-1}\\
n^{1-\beta},&|\theta|<0.01n^{-1}
\end{array}\right.
\]
uniformly in $n$.
\end{lemma}
\begin{proof}
For $|\theta|\,{<}\,0.01n^{-1}$, the estimate is obtained by taking the absolute values in the sum.
For $|\theta|\!>\!0.01n^{-1}$,
\[
|A_n'|\lesssim\bigg|\sum^n_{j=1}j^{-\beta}e^{ij\theta}\bigg|+1\,\,.
\]
The estimates \eqref{sin-a} and \eqref{cos-a} along with the trivial
bounds on the integrals involved yield the statement of the
Lemma.\end{proof} \bigskip

{\bf Remark.} Notice that, as $n$ is large enough, the estimates
obtained in Lemmas \ref{poly2} and \ref{poly1} (except for the
bounds on the imaginary parts that are violated near $z=-1$) can be
extended from the small arc $|\theta|<\upsilon$ to the whole circle
using the uniform convergence of the corresponding Taylor expansions
outside any fixed arc $|\theta|<\upsilon$.
\bigskip

\smallskip
Here we give the proof to the Theorem \ref{neva-ger} (check the
paper \cite{nt} for the related questions).

\begin{proof}{\it(of the Theorem \ref{neva-ger}).}
Since $\sigma$ belongs to the Steklov class, it  belongs to the
Szeg\H{o} class as well and thus the Schur coefficients
$\{\gamma_n\}\in \ell^2$. In particular, $\gamma_n\to 0$  and
$\rho_n\to 1$. Take $z\in\mathbb T$, divide the second equation in
\eqref{srecurs} by $\phi_n^*$, and take the absolute value to get
\[
\left|\frac{\phi^*_{n+1}}{\phi_n^*}\right|=|\rho_n|^{-1}\cdot
\left|1-\gamma_n z\frac{\phi_n}{\phi_n^*}\right|
\]
Since $|\phi_n|=|\phi_n^*|$ for $z\in \mathbb{T}$, we have
\[
\sup_{z\in\mathbb T}\left|\left|
\frac{\phi_{n+1}(z,\sigma)}{\phi_n(z,\sigma)}\right| -1\right|\to 0,
\quad n\to \infty
\]
Iterating, one has
\begin{equation}\label{blizkoo}
\sup_{z\in \mathbb{T}}\left|\left|
\frac{\phi_{n+j}(z,\sigma)}{\phi_n(z,\sigma)}\right| -1\right|\to 0,
\quad n\to \infty , \quad j\,\, {\rm is \,fixed}
\end{equation}
Now, suppose \eqref{malenkoeo} fails. Then, there is
$\{m_n\}\subseteq \mathbb{N}$ and $\{z_n\}\in \mathbb{T}$ such that
\[
|\phi_{m_n}(z_n)|>C\sqrt{m_n}
\]
So, given arbitrary large fixed $K$, \eqref{blizkoo} implies
\[
|\phi_{m_n+k}(z_n)|>0.9C\sqrt{m_n}
\]
for every $k: |k|\le K$ and $n>n(K)$.  In particular,
\[
\sum_{j=0}^{m_n} |\phi_j(z_n)|^2>(0.9)^2C^2Km_n, \quad n>n(K)
\]
This, however, contradicts \eqref{chezaro} as $K$ is arbitrarily large.
\end{proof}\vspace{0.5cm}

\part*{Appendix B.}\bigskip
In this section, we control the phases of various functions we used
in the text. Let us start with $\phi$, the phase of
$Q_m(e^{i\theta})$, for $|\theta|\!<\!\upsilon$, where $\upsilon$ is
some small, positive, and fixed number.
\begin{lemma}\label{fazka1}
For any $\theta\!\in\!(-\upsilon,\upsilon)$, we have
\[
|\phi'(\theta)|\lesssim m\,\,.
\]
\end{lemma}
\begin{proof}Recall that (see \eqref{mult-mult})
\begin{equation}\label{mult-mult1}
Q_m(z)=\exp\left(\frac1{2\pi}\int^\pi_{-\pi}C(z,e^{i\xi})\ln|Q_m(e^{i\xi})|d\xi\right)~,\quad z\!\in\!\mathbb D
\end{equation}
and $\phi(\theta)=\arg Q_m(e^{i\theta})$, i.e.,
\[
\phi(\theta)=\Im\left(\frac1{2\pi}\int^\pi_{-\pi}C(e^{i\theta},e^{i\xi})\ln|Q_m(e^{i\xi})|d\xi\right)
\]
where, as before,
\[
C(e^{i\theta},e^{i\xi})=\frac{e^{i\xi}+e^{i\theta}}{e^{i\xi}-e^{i\theta}}
\]
and the integral is taken in principal value. Thus,
\[
\phi(\theta)=-\frac1{2\pi}\int^\pi_{-\pi}\frac{\cos((\xi-\theta)/2)}{\sin((\xi-\theta)/2)}\ln|Q_m(e^{i\xi})|d\xi
\]
This amounts to controlling the Hilbert transform of $\ln|Q_m(e^{i\xi})|$ since
\[
\frac{\cos(\xi/2)}{\sin(\xi/2)}=\frac2\xi+O(\xi),\quad\left(\frac{\cos(\xi/2)}{\sin(\xi/2)}\right)'=-\frac2{\xi^2}+O(1)\,\,.
\]
From the periodicity,
\[
\phi(\theta)=-\frac1{2\pi}\int^{\theta+\pi}_{\theta-\pi}\frac{\cos((\xi-\theta)/2)}{\sin((\xi-\theta)/2)}\ln|Q_m(e^{i\xi})|d\xi\,\,.
\]
Changing the variables, we have
$$
\phi(\theta)=-\frac1{2\pi}\int^\pi_{-\pi}\frac{\cos(\xi/2)}{\sin(\xi/2)}\ln|Q_m(e^{i(\xi+\theta)})|d\xi
=-\frac1{2\pi}\int^\pi_{-\pi}\frac{\cos(\xi/2)}{\sin(\xi/2)}\cdot\frac12\big(D_m(\xi\!+\!\theta)+\ln
m\big)d\xi
$$
where
\[
D_m(\xi)=\ln\left( \cal{G}_m(\xi) +|R_{(m,\alpha/2)}(e^{i\xi})|^2\right)+\ln
m^{-1}\,\,.
\]
Then, $$\phi'(x)
=-\frac1{2\pi}\int^\pi_{-\pi}\frac{\cos(\xi/2)}{\sin(\xi/2)}\frac12D_m'(\xi\!+\!x)d\xi\,\,.
$$
We then use the Taylor expansion for
\[
\frac{\cos(\xi/2)}{\sin(\xi/2)}
\]
and integrate by parts using the periodicity to approximate the
integral  by the Hilbert transform
$$
\Big|\phi'(x)\!+\!\frac1{2\pi}\!\int^\pi_{-\pi}\!\!\frac{D_m'(\xi\!+\!x)}\xi
d\xi\Big|= \left|\int_{-\pi}^\pi
\left(\frac{\cos(\xi/2)}{\sin(\xi/2)}-\frac{2}{\xi}\right)'\ln\left(
\cal{G}_m(\xi) +|R_{(m,\alpha/2)}(e^{i\xi})|^2\right)d\xi\right|
$$
$$
\lesssim \int_{-\pi}^\pi |\ln\left( \cal{G}_m(\xi)
+|R_{(m,\alpha/2)}(e^{i\xi})|^2\right)|d\xi\lesssim 1\,\,.
$$
The last inequality follows from
\[
\ln x<x,\quad x>1\quad\implies\qquad|\ln x|<2x+\ln\frac1x
\]
\[
\int_{\mathbb T}\Bigl(\cal{G}_m(\xi)
+|R_{(m,\alpha/2)}(e^{i\xi})|^2\Bigr)d\xi\lesssim 1
\]
and
\[
\cal{G}_m(\xi) +|R_{(m,\alpha/2)}(e^{i\xi})|^2\gtrsim 1, \quad m>m_0
\]
(see Lemma \ref{poly2} for the estimates on $R_{(m,\alpha/2)}$).
 Therefore, if $x\in (-\upsilon,\upsilon)$, then
\[
|\phi'(x)|\lesssim\left|\int^\pi_{-\pi}\frac{D_m'(\xi+x)}{\xi}d\xi\right|+1
\]
and
\[
|\phi'(x)|\lesssim 1+m\left|\int_{-m\pi}^{m\pi}\frac{M_m'(t+\widehat
x)}{tM_m(t+\widehat x)}dt\right|,\quad M_m(t)=\exp
D_m(t/m),\quad\widehat x=mx
\]
$$M_m(t)=\frac1m\Big|Q_m\Big(\frac tm\Big)\Big|^2$$
For $M_m$,
\begin{eqnarray}\label{em-en}
M_m(t)=\frac{\sin^2 (t/2)}{m^2\sin^2(t/(2m))} +\frac{\cos^2
(t/2)}{2m^2\sin^2((t-\pi)/(2m))}+
\\
\frac{\cos^2(t/2)}{2m^2\sin^2((t+\pi)/(2m))}+
m^{-1}|R_{(m,\alpha/2)}(e^{it/m})|^2 \nonumber
\end{eqnarray}
due to \eqref{feya} and \eqref{sdvig}. Thus, we only need to show
that
\[
I_1(\widehat x)=\left|\int_{-1}^1 \frac{M_m'(t+\widehat{x})}{tM_m(t+\widehat x)}dt\right|\lesssim 1
\]
and
\[
I_2(\widehat x)=\left|\int_{1<|t|<\pi m} \frac{M_m'(t+\widehat{x})}{tM_m(t+\widehat x)}dt\right|\lesssim 1
\]
uniformly in $\widehat x\in[-m\upsilon,m\upsilon]$.

Let $\cal J_n(\xi)$ denote the sum of the first three terms in
\eqref{em-en}. Then we can rewrite it as follows
\begin{equation}\label{uh-uh}\begin{aligned}
\cal J_n(\xi)=\sin^2\frac\xi2\left(\frac4{\xi^2}+\frac1{m^2}G\Big(\frac\xi{2m}\Big)\right)+
\frac12&\cos^2\frac\xi2\left(\frac4{(\xi\!-\!\pi)^2}+\frac1{m^2}G\Big(\frac{\xi\!-\!\pi}{2m}\Big)\right)+\\+\,
\frac12&\cos^2\frac\xi2\left(\frac4{(\xi\!+\!\pi)^2}+\frac1{m^2}G\Big(\frac{\xi\!+\!\pi}{2m}\Big)\right)
\end{aligned}\end{equation}
where
\[
G(x)=\frac1{\sin^2 x}-\frac1{x^2}
\]
is positive infinitely smooth function defined on $(-\pi,\pi)$ and
$G(x)\sim 1$ on $[-a,a]\subset(-\pi,\pi)$.\smallskip

Let us start with $I_2$ and take $t: |t|<\pi m$. Therefore, for
$\xi=t+\widehat x$, we have $|\xi|<(\pi+\upsilon)m$.
\smallskip

\noindent We will write a lower bound for $\cal J_n(\xi)$ for large
and for small $\xi$.

\noindent For large $\xi$, i.e., $|\xi|\!\ge\!c_1\!>\!\pi$,
$\frac{|\xi|}m\!\le\!c_2\!<\!2\pi$, we have:
$$
\cal
J_n(\xi)=\frac4{\xi^2}\sin^2\frac\xi2+\Big(\frac4{\xi^2}+O(|\xi|^{-3})\!\Big)\cos^2\frac\xi2+
O^*\Big(\frac{1}{m^2}\Big)=\frac4{\xi^2}+O(|\xi|^{-3})+O^*\Big(\frac1{m^2}\Big)\,\,.
$$

\noindent For $\xi\!\in\![-a,a]$ with fixed $a$, we get
\begin{align*}\cal J_n(\xi)&=
\frac4{\xi^2}\sin^2\frac\xi2+\left(\frac2{(\xi\!-\!\pi)^2}+\frac2{(\xi\!+\!\pi)^2}\right)
\cos^2\frac\xi2+O\Big(\frac1{m^2}\Big)\ge\\&\ge\min\!\left(\frac4{\xi^2},
\frac2{(\xi\!-\!\pi)^2}+\frac2{(\xi\!+\!\pi)^2}\right)\Big(\sin^2\frac\xi2+\cos^2\frac\xi2\Big)
+O\Big(\frac1{m^2}\Big)\gtrsim1
\end{align*}

Then, for $|\xi|<(\pi+\upsilon)m$, we have
\begin{equation}\label{shaliy}
M_m(\xi)\gtrsim\min(1,|\xi|^{-2})+m^{-1}|R_{(m,\alpha/2)}(e^{i\xi/m})|^2\,\,.
\end{equation}
For the derivative of $M_m(\xi)$, the representation \eqref{uh-uh} gives an upper bound
$$
\cal
J_n'(\xi)=-\frac8{\xi^3}\sin^2\frac\xi2-\left(\frac4{(\xi\!-\!\pi)^3}+
\frac4{(\xi\!+\!\pi)^3}\right)
\cos^2\frac\xi2+O\Big(\frac{\|G'\|_\infty}{m^3}\Big)+$$$$+\frac12\sin\xi\Big[
\frac4{\xi^2}-\frac2{(\xi\!-\!\pi)^2}-
\frac2{(\xi\!+\!\pi)^2}+\frac1{m^2}G\Big(\frac\xi{2m}\Big)
-\frac1{2m^2}G\Big(\frac{\xi\!-\!\pi}{2m}\Big)-
\frac1{2m^2}G\Big(\frac{\xi\!+\!\pi}{2m}\Big)\Big]\,\,.
$$
For large $|\xi|$, we can write
$$
\cal J_n'(\xi)=-\frac8{\xi^3}+O(|\xi|^{-5})+
O\Big(\frac1{m^3}\Big)+O(|\xi|^{-4})+O\Big(\frac{1}{m^4}\|G''\|_\infty\Big)\,\,.
$$
For $\xi\in [-a,a]$ with fixed $a$, we again use the smoothness of
$\cal J_n$.
\begin{gather*}
|\cal J_n'(\xi)|<\Big\|-\frac8{\xi^3}\sin^2\frac\xi2-
\left(\frac4{(\xi\!-\!\pi)^3}+\frac4{(\xi\!+\!\pi)^3}\right)
\cos^2\frac\xi2+\frac12\sin\xi\Big(\frac4{\xi^2}-
\frac2{(\xi\!-\!\pi)^2}-\frac2{(\xi\!+\!\pi)^2}\Big)\Big\|_\infty
+\\+O\Big(\frac{\|G'\|_\infty}{m^3}\Big)+O\Big(\frac{\pi^2}{8m^4}\|G''\|_\infty\Big)\sim1\,\,.
\end{gather*}
Combining these results, we obtain
\begin{equation}\label{last-term}
|M_m'(\xi)|\lesssim
\frac{1}{(1+|\xi|)^3}+\Bigl|(m^{-1}|R_{(m,\alpha/2)}(e^{i\xi/m})|^2)'\Bigr|,
\quad |\xi|\lesssim m\,\,.
\end{equation}
\smallskip
First, consider  $\xi: 1<|\xi|<(\pi+\upsilon) m$. The Lemma
\ref{poly2} and \eqref{shaliy} give
\[
M_m(\xi)\gtrsim \frac{1}{\xi^2}+m^{-1}\left|\frac{\xi}{m}\right|^{-\alpha}\,\,.
\]
Now, it is sufficient to use Lemmas \ref{poly2} and  \ref{derider}
to bound the last term in \eqref{last-term} as
\[
\Bigl|(m^{-1}|R_{(m,\alpha/2)}(e^{i\xi/m})|^2)'\Bigr|\le\frac2{m^2}|B'_m||B_m|\Big|_{\theta=\frac\xi
m}\lesssim\frac1{m^2}\left|\frac\xi m\right|^{-\alpha/2}\frac{m^{1+\alpha/2}}{|\xi|}\,\,.
\]
Combining these bounds, we have
\begin{equation}\label{modin}
\left|\frac{M_m'(\xi)}{M_m(\xi)}\right|\lesssim\frac{\displaystyle
\frac1{|\xi|^3}+\frac{m^{\alpha-1}}{|\xi|^{1+\alpha/2}}}{\displaystyle
\frac1{\xi^2}+\frac{m^{\alpha-1}}{|\xi|^\alpha}}\le\frac{\displaystyle
\frac1{|\xi|^3}}{\dfrac1{\xi^2}}+
\frac{\dfrac{m^{\alpha-1}}{|\xi|^{1+\alpha/2}}}{\dfrac{m^{\alpha-1}}{|\xi|^\alpha}}
\lesssim|\xi|^{-1}+|\xi|^{\alpha/2-1}\,\,
\end{equation}
for $1<|\xi|<(\pi+\upsilon)m$.
\smallskip

For $\xi: |\xi|<1$, the analogous estimates give
\begin{equation}\label{modin1}
\left|\frac{M_m'(\xi)}{M_m(\xi)}\right|\lesssim
\frac{1+m^{\alpha-1}}{1+m^{\alpha-1}}\lesssim 1\,\,.
\end{equation}
Combining \eqref{modin} and \eqref{modin1}, we get
\begin{equation}\label{modin11}
\left|\frac{M_m'(\xi)}{M_m(\xi)}\right|\lesssim
(|\xi|+1)^{-1}+(|\xi|+1)^{\alpha/2-1}
\end{equation}
which holds uniformly in $\xi:|\xi|<(\pi+\upsilon)m$. Now, the
Cauchy-Schwarz inequality implies the bound for $I_2$
\[
|I_2(\widehat x)|\le\left(\int_{1<|t|<\pi m}\frac{dt}{t^2}\right)^{1/2}
\left(\int_{|\xi|<(\pi+\upsilon)m}\left|\frac{M'(\xi)}{M(\xi)}\right|^2d\xi\right)^{1/2}\lesssim 1\,\,.
\]
\smallskip

Consider $I_1$. Apply the Mean Value Formula to rewrite it as
\[
|I_1(\widehat x)|=\left|\int_{-1}^1\frac1t\left(\frac{M_n'(\widehat x)}{M_n(\widehat x)}
+t\left(\frac{M_n'(\xi)}{M_n(\xi)}\right)'_{\xi=\xi_{\widehat x,t}}\right)dt\right|\lesssim
\left\|\frac{M_m''}{M_m}\right\|_\infty+\left\|\frac{M_m'}{M_m}\right\|^2_\infty\,\,.
\]
The second term was estimated in \eqref{modin1} so we only need to
control the first one. We use \eqref{uh-uh} and \eqref{shaliy} to
get
\[
\left|\frac{M_m''(\xi)}{M_m(\xi)}\right|\lesssim\frac{\displaystyle
\frac1{(|\xi|+1)^2}+\left|\left(m^{-1}|R_{(m,\alpha/2)}(e^{i\xi/m})|^2\right)''\right|}{\displaystyle
\frac1{1+\xi^2}+m^{-1}\left|R_{(m,\alpha/2)}(e^{i\xi/m})\right|^2}\,\,.
\]
The estimates from the Lemmas \ref{poly2} and \ref{derider}  in
Appendix A can now be used as follows. We have
\[
\left|\frac{M_m''(\xi)}{M_m(\xi)}\right|\lesssim 1+2\left|
\frac{(R_{(m,\alpha/2)}(e^{i\xi/m}))'}{R_{(m,\alpha/2)}(e^{i\xi/m})}\right|^2+2\left|
\frac{(R_{(m,\alpha/2)}(e^{i\xi/m}))''}{R_{(m,\alpha/2)}(e^{i\xi/m})}\right|
\]
since $(R\overline R)''=R''\overline R+2R'\overline R'+R\overline
R''$. For $\xi: 1<|\xi|<\upsilon m+1$, one gets
\[
\left|
\frac{(R_{(m,\alpha/2)}(e^{i\xi/m}))'}{R_{(m,\alpha/2)}(e^{i\xi/m})}
\right|\lesssim |\xi|^{\alpha/2-1}, \quad
\left|
\frac{(R_{(m,\alpha/2)}(e^{i\xi/m}))''}{R_{(m,\alpha/2)}(e^{i\xi/m})}
\right|\lesssim |\xi|^{\alpha/2-1}\,\,.
\]
For $\xi: |\xi|<1$, we have
\[
\left|
\frac{(R_{(m,\alpha/2)}(e^{i\xi/m}))'}{R_{(m,\alpha/2)}(e^{i\xi/m})}
\right|\lesssim 1,\quad
\left|
\frac{(R_{(m,\alpha/2)}(e^{i\xi/m}))''}{R_{(m,\alpha/2)}(e^{i\xi/m})}
\right|\lesssim 1\,\,.
\]
This gives a bound
\[
 \left\| \frac{M_m''}{M_m}\right\|_\infty\lesssim 1
\]
which ensures
\[
|I_1(\widehat x)|\lesssim 1
\]
uniformly in $\widehat x\in[-m\upsilon,m\upsilon]$. The proof is
finished.
\end{proof}

\bigskip


\bigskip

In the next Lemma, we will prove \eqref{1976-1} and \eqref{1976-2}.
\begin{lemma}\label{faza-5}
For the functions $H_n$ and $\widetilde F$ introduced in part 2, we
have the following bounds
\begin{equation}\label{1976-1-n}
\left|\partial_\theta\left(\frac{2+H_n(1+\~F)}{2+\overline
H_n(1+\overline{\~F})}\right)\right|<0.01n
\end{equation}
and
\begin{equation}\label{1976-2-n}
\left|\partial_\theta\left(\frac{2+\overline
H_n(1-\~F)}{2+H_n(1-\overline{\~F})}\right)\right|<0.01n
\end{equation}
provided that $\rho\ll 1$ and $n\gg 1$.
\end{lemma}
\begin{proof}
 We will only prove
\eqref{1976-1-n} as the other bound is similar. We have
\[
2+H_n(1+\~ F)=2+(1-z)(1-0.1R_{(m,-(1-\alpha))}(z))(1+\~
C_n\rho(1+\e_n-z)^{-1}+\~ C_n(1+\e_n-z)^{-\alpha})\,\,.
\]
First, notice that $\~ C_n\to 1$ as $\rho\to 0$ and $\e_n\to 0$ as
follows from \eqref{small-w}. The formula \eqref{eto-ash} for $H_n$
implies that
\[
|H_n(z)|<C(\alpha)|1-z|, \quad z\in \mathbb{T}
\]
Then,
\[
\max_{z\in \mathbb T}\Re \Bigl(\~ C_n H_n
\rho(1+\e_n-z)^{-1}\Bigr)\to 0, \quad \rho\to 0\,\,.
\]
Finally,
\[
|H_n\~ C_n (1+\e_n-z)^{-\alpha}|<\~
C_n|1+\e_n-z|^{1-\alpha}(1+0.1|1-z|^{1-\alpha}+o(1))
\]
where $o(1)\to 0$ as $m\to \infty$ because
$R_{(m,-(1-\alpha))}(z)\to (1-z)^{1-\alpha}$ uniformly on the
circle. Therefore,
\[
|H_n\~ C_n (1+\e_n-z)^{-\alpha}|<2^{1-\alpha}\Bigl(1+\frac{
2^{1-\alpha}}{10}\Bigr)+o(1)<\sqrt 2\Bigl(1+\frac{\sqrt
2}{10}\Bigr)+o(1)<1.8
\]
since $\alpha\in (0.5,1)$. Thus,
\[
\Re (2+H_n(1+\~ F))>0.2
\]
and so
\[
|2+H_n(1+\~ F)|>0.2
\]
for all $z\in \mathbb T$. Then,
\[
\left|\partial_\theta \left(\frac{2+H_n(1+\~ F)}{2+\overline
H_n(1+\overline{\~ F})}\right)\right|\le
2\left|\frac{\partial_\theta\Bigl(H_n(1+\~ F)\Bigr)}{2+H_n(1+\~
F)}\right|<10\Bigl|\partial_\theta\Bigl(H_n(1+\~ F)\Bigr)\Bigr|\,\,.
\]
We have
\[
\partial_\theta\Bigl(H_n(1+\~
F)\Bigr)=H_n'(1+\~ F)+H_n\~ F'\,\,.
\]
The explicit expressions for $H_n$ and $\~ F$ give
\[
|H_n\~ F'|<C(\alpha)
|1-z|\left(\rho|1+\e_n-z|^{-2}+|1+\e_n-z|^{-1-\alpha}\right)<C(\alpha)\bigl(
\rho n+n^{\alpha}\bigr), \quad z\in \mathbb T
\]
and
\[
|H_n'(1+\~ F)|<C(\alpha)(1+|\~ F|)+0.1|(1-z)(1+\~ F)|\cdot
|R'_{(m,-(1-\alpha))}(z)|\,\,.
\]
The Lemma \ref{der-der} implies that the third term is bounded by
$Cn^{\alpha}$ and we have the bound
\[
|H_n'(1+\~ F)|<C(\alpha)(1+\rho n+n^\alpha)
\]
uniformly over $\mathbb T$. Making $\rho$ small and $n$ large
finishes the proof of \eqref{1976-1-n}. \end{proof} {\bf Remark.}
The estimates in the Lemma above are valid for $\rho\in (0,\rho_0)$
and $n>n_0$ where $\rho_0$ and $n_0$ both depend on $\alpha$.

\vspace{0.5cm}


\bigskip




\begin{thebibliography}{99}
\bibitem{murman}M. U. Ambroladze, On the possible rate of growth of polynomials that are orthogonal
with a continuous positive weight (Russian), Mat. Sb. 182 (1991), no. 3, 332--353;
English translation in: Math. USSR-Sb. 72 (1992), no. 2, 311--331.

\bibitem{ap1} A. I. Aptekarev, V. S. Buyarov, I. S. Dehesa,
Asymptotic behavior of $L^p$--norms and the entropy for general
orthogonal polynomials, Russian Acad. Sci. Sb. Math.  1995, 82 (2),
373--395.



\bibitem{ap2} A. I. Aptekarev, J. S. Dehesa, A. Mart\'inez-Finkelshtein,
Asymptotics of orthogonal polynomial's entropy. J. Comput. Appl.
Math. 233 (2010), no. 6, 1355--1365.

\bibitem{askey}  G. E. Andrews, R. Askey,
R. Roy, ``Special Functions", Cambridge University Press, 2000.


\bibitem{ap3}
B. Beckermann, A. Mart\'inez-Finkelshtein, E. A. Rakhmanov, F. Wielonsky,
Asymptotic upper bounds for the entropy of orthogonal polynomials in the Szeg\H o class.
J. Math. Phys. 45 (2004), no. 11, 4239--4254.

\bibitem{bernstein} S. Bernstein, Sur les polynomes orthogonaux relatifs $\grave{\rm {a}}$ un segment
fini, Journal de Mathem$\acute{\rm{a}}$tiques, (9), 9 (1930),
pp.~127--177; 10 (1931), pp.~219--286.

\bibitem{den-pams} S. Denisov, On the size of the polynomials orthonormal on the unit
 circle with respect to a measure which is a sum of the Lebesgue measure and  $p$ point
 masses, preprint.

\bibitem{dk} S. Denisov, S. Kupin,
On the growth of the polynomial entropy integrals for measures in
the Szeg\H{o} class, Advances in Mathematics, Vol. 241, 2013,
18--32.

\bibitem{duren} P. Duren, Theory of $H^p$ spaces. Dover
publications, Mineola, New York, 2000.

\bibitem{5} Ya. L. Geronimus,
Polynomials orthogonal on the circle and on the interval,
GIFML, Moscow, 1958 (in Russian);
English translation:  International Series of Monographs on Pure and Applied Mathematics, Vol. 18
Pergamon Press, New York-Oxford-London-Paris, 1960.


\bibitem{Ger1} Ya. L. Geron─лmus,
Some estimates of orthogonal polynomials and the problem of Steklov.
Dokl. Akad. Nauk SSSR, 236 (1977), no. 1, 14--17.

\bibitem{Ger2} Ya. L.  Geronimus,
The relation between the order of growth of orthonormal polynomials and their weight function.
Mat. Sb. (N.S.) 61 (103), 1963, 65--79.

\bibitem{Ger3} Ya. L. Geronimus,
On a conjecture of V. A. Steklov.
Dokl. Akad. Nauk SSSR, 142, 1962, 507--509.


\bibitem{Gol} B. L. Golinskii, The problem of V. A. Steklov in the theory of orthogonal polynomials. Mat. Zametki, 15 (1974), 21--32.

\bibitem{murad} M.E.H. Ismail, Classical and quantum orthogonal
polynomials in one variable. Encyclopedia of Mathematics, 98,
Cambidge University Press, 2005.

\bibitem{doron2} A. Kroo, D. Lubinsky,  Christoffel functions
and universality in the bulk for multivariate orthogonal
polynomials.  Canad. J. Math. 65 (2013), no. 3, 600--620.


\bibitem{doron1} D. Lubinsky,  A new approach to universality
limits involving orthogonal polynomials.  Ann. of Math. (2) 170
(2009), no. 2, 915--939.


\bibitem{mnt} A. Mate, P. Nevai, V. Totik,
 Szeg\H{o}'s extremum problem on the unit circle. Ann. of Math. (2)
134 (1991), no. 2, 433--453.




\bibitem{neva} P. Nevai,
Orthogonal polynomials. Mem. Amer. Math. Soc. 18 (1979), no. 213.



\bibitem{nt} P. Nevai, J. Zhang, V. Totik,
Orthogonal polynomials: their growth relative to their sums. J.
Approx. Theory 67 (1991), no. 2, 215--234.



\bibitem{PS} G. P\'{o}lya, G. Szeg\H{o}, Problems and theorems in
Analysis I, Berlin-Heidelberg-New York, Springer, 1978.


\bibitem{3} E. A. Rahmanov,
On Steklov's conjecture in the theory of orthogonal polynomials,
Matem. Sb., 1979, 108(150), 581--608;
English translation in:  Math. USSR, Sb., 1980, 36, 549--575.

\bibitem{4} E. A. Rahmanov,
Estimates of the growth of orthogonal polynomials whose weight is bounded away from zero,
Matem. Sb., 1981, 114(156):2, 269--298;
English translation in: Math. USSR, Sb., 1982, 42, 237--263.

\bibitem{6}  G.~Szeg\H{o},
{Orthogonal Polynomials}, Amer.\ Math.\ Soc.\ Colloq.\ Publ.\
\textbf{23}, Providence RI, 1975 (fourth edition).



\bibitem{sim1} B. Simon,
Orthogonal polynomials on the unit circle,
volumes 1 and 2, AMS 2005.


\bibitem{1} V. A. Steklov,
Une methode de la solution du probleme de development des fonctions en series de polynomes
de Tchebysheff independante de la theorie de fermeture,
Izv. Rus. Ac. Sci., 1921, 281--302, 303--326.

\bibitem{2} P. K. Suetin,
V. A. Steklov's problem in the theory of orthogonal polynomials,
Itogi Nauki i Tech. Mat. Anal., VINITI, 1977, 15, 5--82 ;
English translation in: Journal of Soviet Mathematics, 1979, 12(6), 631--682.

\bibitem{tot} V. Totik,
Christoffel functions on curves and domains. Trans. Amer. Math. Soc.
362 (2010), no. 4, 2053--2087.

\bibitem{tot2} V. Totik,  Asymptotics for Christoffel functions for
general measures on the real line.  J. Anal. Math. 81 (2000),
283--303.

\bibitem{zygmud} A. Zygmund, Trigonometric series. Third edition,
Cambridge University Press, 2002.



\end{thebibliography}
\end{document}